\documentclass[a4paper]{amsart}
\usepackage{latexsym}\usepackage{ifthen}\usepackage[leqno]{amsmath}
\usepackage{enumerate}\usepackage{calc}\usepackage{hyphenat}
\usepackage{anyfontsize}
\usepackage{moresize}
\usepackage{mathtools}
\usepackage{bbm}
\usepackage{amstext,amsbsy,amsopn,amsthm,amsgen, amsmath}
\usepackage{amsfonts,amscd,amsxtra,upref}
\usepackage{graphicx,varwidth}
\usepackage{wrapfig}
\usepackage{pdfsync}
\usepackage[margin=1in,headheight=13.6pt]{geometry}
\usepackage{comment}
\usepackage{mathtools}
\usepackage{mathdots}
\usepackage{tikz}
\usepackage{pdflscape}

\usepackage{multirow,multicol, makecell, booktabs}
\usepackage{float}
\usepackage{array,amsmath}

\usepackage{afterpage}

\usepackage{expl3,xparse}

\usepackage{slashbox}
\usepackage{tabularx}
\usepackage{stackrel}
\usepackage{caption, float}
\usepackage{hhline}
\usepackage{dirtytalk}

\usepackage{pdfpages}

\usepackage{picture,slashbox}

\usepackage{epstopdf}
\usepackage{dsfont}
\usepackage{hyperref}
\usepackage{mathrsfs}\usepackage{euscript}\usepackage{amssymb}
\swapnumbers
\theoremstyle{plain}

\usepackage[
backend=bibtex,
style=alphabetic,
natbib=true,
sortlocale=en_US,
url=true, 
doi=true,
eprint=true,
]{biblatex}
\newtheorem{Thm}{Theorem}[section]
\newtheorem*{thmnonum}{Theorem}

\newtheorem{Lem}[Thm]{Lemma}

\newtheorem{Cor}[Thm]{Corollary}

\theoremstyle{definition}
\newtheorem{Def}[Thm]{Definition}

\theoremstyle{remark}

\numberwithin{equation}{section}

\usepackage{scalerel}
\usepackage{stackengine,wasysym}
\newcommand\tylda[1]{\ThisStyle{%
		\setbox0=\hbox{$\SavedStyle#1$}%
		\stackengine{-.1\LMpt}{$\SavedStyle#1$}{%
			\stretchto{\scaleto{\SavedStyle\mkern.2mu\AC}{.5150\wd0}}{.6\ht0}%
		}{O}{c}{F}{T}{S}%
}}

\usepackage{lipsum}

\newcommand{\myData}[1][]{
\author[D.\ Burek]{Dominik Burek}
\address{D.\ Burek{}\\{}
Faculty of Mathematics and Computer Science\\{}Jagiellonian University,
\L{}ojasiewicza 6\\{}30-348 Krak\'{o}w\\{}Poland}
\email{dominik.burek@uj.edu.pl}
}

\newcommand{\myCurrentData}[1][]{
\address{Current address: 
Leibniz Universit{\"a}t Hannover\\{}Fakult{\"a}t f{\"u}r Mathematik und Physik \\{} Welfengarten 1\\{}30167 Hannover}}

\newcommand{\ZZ}{\mathbb{Z}}
\newcommand{\CC}{\mathbb{C}}

\newcommand{\PP}{\mathbb{P}}
\newcommand{\OO}{\mathcal{O}}
\newcommand{\FF}{\mathbb{F}}

\newcommand{\diag}{\operatorname{diag}}

\newcommand{\QQ}{\mathbb{Q}}
\newcommand{\et}{\operatorname{\acute{e}t}}

\newcommand{\id}{\operatorname{id}}
\newcommand{\Fix}{\operatorname{Fix}}
\newcommand{\Frob}{\operatorname{Frob}}

\newcommand{\kum}{\operatorname{Kum}}

\newcommand{\orb}{\operatorname{orb}}

\newcommand*\dd{\mathop{}\!\mathrm{d}}
\newcommand{\age}{\operatorname{age}}
\newcommand*\diff{\mathop{}\!\mathrm{d}}
\newcommand{\linia}{\rule{\linewidth}{0.3mm}}
\newcommand{\overbar}[1]{\mkern 1.5mu\overline{\mkern-1.5mu#1\mkern-1.5mu}\mkern 1.5mu}

\renewcommand{\bar}{\overbar}

\begin{document}

	\title{Zeta function of some Kummer Calabi-Yau 3-folds}
	\myData
	
	\begin{abstract} We compute Hodge numbers and zeta function of a Kummer Calabi-Yau $3$-folds introduced by M. Andreatta and J. Wi{\'s}niewski in \cite{Wisa} and investigated by M. Donten-Bury in \cite{Donten}.		 \end{abstract}
	
	\thanks{Data sharing not applicable to this article as no datasets were generated or analysed during the current study.}

	\maketitle
	
	
\section{Introduction}

The classical Kummer construction involves an abelian surface $A$ over $\CC$ and an involution $i\colon A\to A.$ The quotient $X:=A/i$ is a normal surface having $16$ rational singularities of type $A_{1}.$ The minimal resolution of singularities $\tylda{X}$ of $X$ is a $K3$ surface.
 
Various ideas have been used to give generalizations of the Kummer construction. In \cite{UenoK} the author describes the generalized Kummer manifold us an algebraic variety $X$ for which there exists an abelian variety $A$ and a generically surjective rational map $A\to X.$ S. Cynk and K. Hulek in \cite{CH} imitated the classical Kummer construction in the sense that
they considered pairs of varieties $X,$ $Y$ with involutions $i_{X}$ and $i_{Y}.$ Then the authors derived a criterion for when the quotient $(X\times Y)/(i_{X}\times i_{Y})$
admits a crepant resolution of singularities giving a Calabi-Yau manifold. In the paper \cite{CynkJNT} Cynk and M. Sch{\"u}tt using Cynk-Hulek generalised Kummer construction and the Weil restrictions discussed resulting Calabi-Yau manifolds over number fields. For $n=3,$ the construction was exhibited in \cite{BorceaCM} in the context of Calabi-Yau threefold with complex multiplication. 

In \cite{Wisa} M. Andreatta and J. Wi{\'s}niewski gave another generalization, $\textup{Kum}_{n}(A,G)$, of the classical Kummer construction by resolving singularities of the quotient of a complex abelian variety $A$ of dimension $n$ by a finite integral matrix group $G$ in $\textup{GL}_{n}(\ZZ)$. Incidentally, holomorphic symplectic manifolds introduced by A. Beauville in \cite{Bowil2} are examples of manifolds $\textup{Kum}_{n}(A,G)$ for some group $G$. Andreatta and Wi{\'s}niewski gave also a new method to compute the cohomology of the constructed Kummer variety. In the case of Beauville's holomorphic symplectic manifolds the cohomology has been computed by using another idea by L. G{\"o}ttsche and W. Soergel in \cite{Gecze}, B. Fantechi and L. G{\"o}ttsche in \cite{Fantechi}. In dimension $3$, under some restrictions on the resolution of singularities and the group action, the resulting variety is a Calabi-Yau threefold. 

M. Donten-Bury in \cite{Donten} investigated the Andreatta-Wi{\'s}niewski construction in the case of $A=E^{3}$ -- a triple product of an elliptic curve $E,$ together with an action of a finite subgroup of $\textup{SL}_{3}(\ZZ)$ on $E^{3}.$ The complete classification of finite subgroups of $\textup{SL}_{3}(\ZZ)$ was given in \cite{tahara}, Donten-Bury classified these subgroups using different idea. Following Andreatta-Wi{\'s}niewski paper, the author gave detailed computations of the Poincar{\'e} polynomial of a crepant resolution $\textup{Kum}_{3}(E,G)$ of the quotient $E^{3}/G.$

The aim of this work is to give a formula for the Hodge numbers and local zeta function of $\textup{Kum}_{3}(E,G)$ using the Chen-Ruan orbifold cohomology theory (\cite{CR}) and the description of the Frobenius action on the orbifold cohomology (\cite{Rose}). The work is motivated by the limited range of examples of Calabi-Yau manifolds which are neither complete intersection in projective spaces nor hypersurfaces in toric space with explicit computed zeta functions. 

 Our main result is: \begin{thmnonum} Hodge numbers and zeta functions of $\kum_{3}(E,G)$ are given in tables from section \ref{results}.  \end{thmnonum}\
 In \cite{Burii1}, \cite{Burii3}, \cite{Burii4} we successfully used that approach in order to compute Hodge numbers and zeta function of manifolds $(X_1\times X_2 \times \ldots \times X_k)/G,$ where $X_i$ ($i=1,2,\ldots, k$) are manifolds of a Calabi-Yau type (i.e. $K_{X_{i}}= \OO_{X_{i}}$) and $G$ is a finite group acting on a product variety $X_1\times X_2 \times \ldots \times X_k.$ All computed examples covered: the famous Borcea-Voisin Calabi-Yau threefolds (\cite{BorceaC}, \cite{V}), $n$-dimensional Calabi-Yau manifold of a Borcea-Voisin type (see \cite{C} for $n=3$, and \cite{Burii4} for an arbitrary dimension $n$).

In sections \ref{chen} and \ref{zeta} we briefly discuss Chen-Ruan cohomology theory and orbifold zeta function. In section 4 we give detailed computations of Hodge numbers and zeta functions of $\textup{Kum}_{3}(E,G)$, for some group $G\in \textup{SL}_{3}(\ZZ)$ of order 24. In the section \ref{results} we give results of computations for all remaining groups.
\subsection*{Acknowledgments} I would like to thank S\l{}awomir Cynk for helpful suggestions and comments. The author is supported by the National Science Center of Poland grant no. 2019/33/N/ST1/01502 and National Science Center of Poland grant no. 2020/36/T/ST1/00265. 

Part of this work was conducted during the stay at the Leibniz Universit{\"a}t in Hannover. I would like to thank the institute for hospitality. Finally, I would like to thank the anonymous referee for their careful reading and helpful suggestions.

\section{Chen-Ruan cohomology}
\label{chen}
In \cite{CR} W. Chen and Y. Ruan introduced a new cohomology theory for an orbifold.   

\begin{Def}\label{Formula} Let $X$ be a projective variety and $G$ be a finite group which acts on $X.$ For a variety $X/G$ define the \textit{Chen-Ruan cohomology} by \begin{equation}\label{orbfor} H_{\textrm{orb}}^{i,j}(X/G):=\bigoplus_{[g]\in\textrm{Conj}(G)}\left(\bigoplus_{U\in \Lambda(g)} H^{i- \age(g),\; j-\age(g)}(U)\right)^{\textup{C}(g)}, \end{equation} where $\textrm{Conj}(G)$ is the set of conjugacy classes of $G$ (we choose a representative $g$ of each conjugacy class), $\textup{C}(g)$ is the centralizer of $g$, $\Lambda(g)$ denotes the set of irreducible connected components of the set fixed by $g\in G$ and $\age(g)$ is the age of the matrix of linearised action of $g$ near a point of $U.$ \par The dimension of $ H_{\textrm{orb}}^{i,j}(X/G)$ will be denoted by $h_{\text{orb}}^{i,j}(X/G).$\end{Def}

An important feature of the Chen-Ruan cohomology is the possibility of computing Hodge numbers of a crepant resolution of singularities of a quotient variety, without referring to an explicit construction of such a resolution i.e.

\begin{Thm}[\cite{TY}, Cor. 3.16] Let $X$ and $X'$ be complete varieties with Gorenstein quotient singularities. Suppose that there are proper birational morphisms $Z\to X$ and $Z\to X'$ such that $K_{Z/X} = K_{Z/X'}.$ Then the orbifold cohomology groups of $X$ and $X'$ have the same Hodge structure.
 \end{Thm}

As a special case we have:

\begin{Thm}\label{TY3}
	Let $G$ be a finite group acting on an algebraic smooth variety $X$. If there exists a crepant resolution $\tylda{X/G}$ of variety $X/G,$ then the following equality holds $$h^{i,j}(\tylda{X/G})=h^{i,j}_{\orb}(X/G).$$
\end{Thm}

For a systematic exposition of the orbifold Chen-Ruan cohomology see \cite{Leida}.

\section{The orbifold zeta function}
\label{zeta}




Let $q$ be a prime power. For a smooth, tame, Deligne-Mumford $\FF_{q}$-stack $\mathcal{X}$ Rose in \cite{Rose} defined \textit{$\ell$-adic Chen-Ruan orbifold cohomology} of $\mathcal{X}\times_{\FF_{q}}\bar{\FF}_{q}$ denoted by $H^{*}_{\textrm{CR}}\left(\mathcal{X}\times_{\FF_{q}}\bar{\FF}_{q}, \QQ_{l}\right)$ (Def. 3.1), introduced the \textit{orbifold Frobenius morphisms} $\Frob_{\orb}$ on $H^{*}_{\textrm{CR}}\left(\mathcal{X}\times_{\FF_{q}}\bar{\FF}_{q}, \QQ_{l}\right)$ (Prop. 1.1) and defined (Def. 6.1) the \textit{orbifold zeta function} by 
\begin{equation*}\label{rosee} Z_{\textrm{CR}}(\mathcal{X}, t):=\det \left(1-\Frob_{\orb}t \mid H^{*}_{\textrm{CR}}\left(\mathcal{X}\times_{\FF_{q}}\bar{\FF}_{q}, \QQ_{l}\right)\right). \end{equation*} It turns out that the orbifold zeta function of a crepant desingularization coincide with the usual zeta function i.e.

\begin{Thm}[\cite{Rose}, Cor. 6.4] Let $\mathcal{X}$ be a proper, smooth, tame Deligne-Mumford stack satisfying the hard Lefschetz condition with trivial generic stabilizer. Suppose $\tylda{X}\to X$ is a crepant resolution of the coarse moduli scheme $X$ of $\mathcal{X},$ then
$$Z_{H^{*}_{\textrm{CR}}}(\mathcal{X}, t)=Z_{q}(\tylda{X},t).$$
where $Z_{q}(\tylda{X},t)$ denotes the classical zeta function $X.$ \end{Thm}

As a special case we have:

\begin{Thm} \label{Rosss}
	Let $G$ be a finite group acting on an algebraic smooth variety $X$. If there exists a crepant resolution $\tylda{X/G}$ of variety $X/G,$ then the following equality holds $$Z_{q}(\tylda{X/G},t)=Z_{H^{*}_{\textrm{CR}}}(X/G, t).$$
\end{Thm}

\section{Quotients of \texorpdfstring{$E^{3}$}{Lg} by a finite subgroup of \texorpdfstring{$\textup{SL}_{3}(\ZZ)$}{Lg}} 

Let $E$ be an elliptic curve, and let $G$ be a finite subgroup of $\textup{SL}_{3}(\ZZ)$. The action of $G$ on $E^{3}$ is regarded as the action of $G$ on the endomorphism ring of $E^{3}.$ The quotient $E^3/G$ admits a crepant resolution of singularities in the sense of Andreatta-Wi{\'s}niewski (\cite{Wisa}) denoted by $\kum_{3}(E,G)$. If $H^{1,0}(E^3)^G=0,$ then $\kum_{3}(E,G)$ is a Calabi-Yau threefold. The Calabi-Yau manifold $\kum_{3}(E,G)$ is defined over any subfield of $\CC$ containing the $j$-invariant of $E$. Note that all finite subgroups $\textup{SL}_{3}(\ZZ)$ were classified in \cite{tahara} as well as in \cite{Donten} by using different idea. 

Denote by $\tau, \eta$ a $2$-torsion and $3$-torsion point of $E$, respectively. 

\subsection{Hodge numbers}
 
In this section we compute Hodge numbers of the variety~$\kum_{3}(E,G_{24.1}),$ where $$G_{24.1}\colon\;\;\left\langle \left(\begin{array}{rrr}0 & 0 & 1\\0 & 1 & 0\\ -1 & 0 & 0 \end{array}\right), \left(\begin{array}{rrr}-1 & 0 & 0\\0 & 0 & -1\\ 0 & -1 & 0 \end{array}\right) \right\rangle\simeq S_{4}$$
is a finite subgroup of $\textup{SL}_{3}(\ZZ)$ of order 24 denoted by \hyperlink{24.1}{$G_{24.1}$}. In the table below w collect all necessary data needed for computation of the Hodge numbers by using Chen-Ruan cohomology \ref{Formula} i.e the fixed loci and the age of representative classes of conjugacy classes of $G_{24.1}$:

$$\textup{Conj}(G_{24.1})=\{[g_0], [g_1], [g_2], [g_3], [g_4]\}$$

\begin{table}[H]
	\begin{center}
		\begin{varwidth}{\textheight}
			\resizebox{0.6\textheight}{!}{
				\begin{tabular}{l||l}
					$g_0=\left(\begin{array}{rrr}1 & 0 & 0 \\ 0 & 1 & 0 \\ 0 & 0 & 1 \end{array}\right)$ &
					\begin{tabular}{@{}l@{}} \setlength{\fboxrule}{1pt}\fcolorbox{white}{white}{$C(g_0)=G_{24.1}$} \\  \setlength{\fboxrule}{1pt}\fcolorbox{white}{white}{$\Fix(g_0)=E^3$} \\ \setlength{\fboxrule}{1pt}\fcolorbox{white}{white}{$\age(g_0)=0$} \\
				    \end{tabular} \\
					\hhline{=::=:}
					$g_1=\left(\begin{array}{rrr}-1 & 0 & 0\\0 & 0 & -1\\ 0 & -1 & 0 \end{array}\right)$ &  
					\begin{tabular}{@{}l@{}}
						\setlength{\fboxrule}{1pt}\fcolorbox{white}{white}{$C(g_1)=\left\langle 
						g_1,\;\left(\begin{array}{rrr}1 & 0 & 0 \\ 0 & -1 & 0 \\ 0 & 0 & -1 \end{array}\right)\right\rangle \simeq \ZZ_2\oplus \ZZ_2$} \\  \setlength{\fboxrule}{1pt}\fcolorbox{white}{white}{$\displaystyle\Fix(g_{1})=\bigcup_{\tau\in E[2]}\{(\tau, -x, x)\}_{x\in E}\simeq \bigcup_{\tau\in E[2]}E - \textup{four copies of $E$}$} \\ \setlength{\fboxrule}{1pt}\fcolorbox{white}{white}{$\age(g_1)=1$} \\
					\end{tabular} \\
					\hhline{=::=:}
					$g_2=\left(\begin{array}{rrr}0 & -1 & 0\\0 & 0 & -1\\ 1 & 0 & 0 \end{array}\right)$ &
					\begin{tabular}{@{}l@{}} 
					\setlength{\fboxrule}{1pt}\fcolorbox{white}{white}{$C(g_2)=\left\langle g_2
						\right\rangle \simeq \ZZ_3 $} \\  \setlength{\fboxrule}{1pt}\fcolorbox{white}{white}{$\displaystyle\Fix(g_2)=\{(x, -x, x)\}_{x\in E}\simeq E - \textup{one copy of $E$}$} \\ \setlength{\fboxrule}{1pt}\fcolorbox{white}{white}{$\age(g_2)=1$} \\
					\end{tabular}\\
					\hhline{=::=:}
					$g_3=\left(\begin{array}{rrr}0 & -1 & 0\\1 & 0 & 0\\ 0 & 0 & 1 \end{array}\right)$ &
					\begin{tabular}{@{}l@{}} 
			        \setlength{\fboxrule}{1pt}\fcolorbox{white}{white}{$C(g_3)=\left\langle g_3
						 \right\rangle \simeq \ZZ_4 $} \\  \setlength{\fboxrule}{1pt}\fcolorbox{white}{white}{$\displaystyle \Fix(g_3)=\{(\tau, \tau, x)\}_{\tau\in E[2],\; x\in E}\simeq \bigcup_{\tau\in E[2]}E - \textup{four copies of $E$}$} \\ \setlength{\fboxrule}{1pt}\fcolorbox{white}{white}{$\age(g_3)=1$} \\
					\end{tabular}\\
					\hhline{=::=:}
					$g_4=\left(\begin{array}{rrr}-1 & 0 & 0\\0 & -1 & 0\\ 0 & 0 & 1 \end{array}\right)$ &
					\begin{tabular}{@{}l@{}}
					 \setlength{\fboxrule}{1pt}\fcolorbox{white}{white}{$C(g_4)=\left\langle g_4,\;
						\left(\begin{array}{rrr}-1 & 0 & 0 \\ 0 & 1 & 0 \\ 0 & 0 & -1 \end{array}\right), \left(\begin{array}{rrr}0 & 1 & 0 \\ -1 & 0 & 0 \\ 0 & 0 & 1 \end{array}\right) \right\rangle \simeq D_8 $} \\  \setlength{\fboxrule}{1pt}\fcolorbox{white}{white}{$\displaystyle\Fix(g_4)=\bigcup_{\tau_1\in E[2]}\bigcup_{\tau_2\in E[2]}\{(\tau_1, \tau_2, x)\}_{x\in E}\simeq \bigcup_{\tau_1\in E[2]}\bigcup_{\tau_2\in E[2]} E - \textup{sixteen copies of $E$}$} \\ \setlength{\fboxrule}{1pt}\fcolorbox{white}{white}{$\age(g_4)=1$} \\ 
					\end{tabular}\\
			\end{tabular}}	
		\end{varwidth}
		\vspace{2mm}
	\end{center}
\end{table}

\textbf{The fixed locus of the action of $g_0=\id.$} As $\Fix(g_0)=E^3$ we have to compute $H^{1,0}(E^3)^{G_{24.1}},$ $H^{1,1}(E^3)^{G_{24.1}},$ $H^{1,2}(E^3)^{G_{24.1}}.$ Denote by $\dd z_{i},$ $\dd\bar{z_{i}}$ pullbacks of invariant form and its complex conjugation by the projection $E^{3}\to E$ on $i$-th component. Then $\{\dd z_{1}, \dd z_{2}, \dd z_{3}\}$ is a basis of the vector space $H^{1,0}(E^3).$ The action of generators of $G_{24.1}$ on $H^{1,0}(E^3)$ is given by matrices 
$$M:=\left(\begin{array}{rrr}0 & 0 & 1\\0 & 1 & 0\\ -1 & 0 & 0 \end{array}\right) \quad \textup{and} \quad N:=\left(\begin{array}{rrr}-1 & 0 & 0\\0 & 0 & -1\\ 0 & -1 & 0 \end{array}\right).$$
Consequently the vector space $H^{1,0}(E^3)^{G_{24.1}}$ is given by the kernel of the $6\times 3$ matrix
\[
\left(\begin{array}{@{}c}
	M-\mathds{1}_{3}\\
\hline
	N-\mathds{1}_{3}
	
\end{array}\right).
\]
In this case 
$$H^{1,0}(E^3)^{G_{24.1}}=\ker \left(
 \begin{array}{rrr}-1 & 0 & 1\\0 & 0 & 0\\ -1 & 0 & -1 \\ -1 & 0 & -1\\0 & 0 & -1\\ 0 & -1 & 0 \end{array}
	\right)=(0).$$
	
Similarly the vector space $H^{1,1}(E^3)^{G_{24.1}}$ has basis $\dd z_{i}\wedge \dd \bar{z_{j}}$, hence the space  $H^{1,1}(E^3)^{G_{24.1}}$ is the kernel of the $18\times 9$ matrix 
\[
\left(\begin{array}{@{}c}
	M\otimes \bar{M}-\mathds{1}_{9}
	\\
	\hline
	N\otimes \bar{N}-\mathds{1}_{9}
	
\end{array}\right), 
\]
where $\otimes$ denotes the Kronecker product of matrices. In our case  $H^{1,1}(E^3)^{G_{24.1}}\simeq \CC.$

Finally $\dd z_{i}\wedge \dd \bar{z_{j}}\wedge \dd \bar{z_{k}}$ forms a basis of $H^{1,2}(E^3),$ hence $H^{1,2}(E^3)^{G_{24.1}}$ is given by the kernel of the matrix 
\[
\left(\begin{array}{@{}c}
	M\otimes \bigwedge^{2}\bar{M}-\mathds{1}_{9}
	\\
	\hline
	N\otimes \bigwedge^{2}\bar{N}-\mathds{1}_{9}
	
\end{array}\right). 
\]
Consequently $H^{1,2}(E^3)^{G_{24.1}}\simeq \CC.$

\textbf{The fixed locus of the action of $g_1.$} Since $$\displaystyle\Fix(g_{1})=\bigcup_{\tau\in E[2]}\{(\tau, -x, x)\}_{x\in E}\simeq \bigcup_{i\in \{1,2,3,4\}}E^{(i)} - \textup{four copies of $E$}$$ and 
$C(g_1)=\left\langle 
g_1, g\right\rangle $ where $$g:=\left(\begin{array}{rrr}1 & 0 & 0 \\ 0 & -1 & 0 \\ 0 & 0 & -1 \end{array}\right)$$ it follows that 
\begin{itemize}
\item $g(E^{(i)})=E^{(i)},$
\item $g$ acts on $H^{0,0}(\Fix(g_{1}))$ as $\diag(1,1,1,1)$, 
\item $g$ acts on $H^{1,0}(\Fix(g_{1}))$ as $\diag(-1,-1,-1,-1).$ 
\end{itemize}
Consequently 
$$H^{0,0}(\Fix(g_{1}))^{C(g_1)}\simeq \CC^{4}\quad \textup{and}\quad H^{1,0}(\Fix(g_{1}))^{C(g_1)}\simeq (0).$$
One can also see that 
$$\Fix(g_{1})/C(g_1)\simeq \textup{four copies of $\PP^{1}$}.$$

\textbf{The fixed locus of the action of $g_2.$} In that case 
$$ \Fix(g_2)=\{(x, -x, x)\}_{x\in E}\simeq E\quad \textup{and} \quad C(g_2)=\left\langle g_2\right\rangle.$$ Therefore $\Fix(g_{2})/C(g_2)\simeq E,$ so $$H^{0,0}(\Fix(g_{2}))^{C(g_2)}\simeq \CC^{}\quad \textup{and}\quad H^{1,0}(\Fix(g_{2}))^{C(g_2)}\simeq \CC^{}.$$

\textbf{The fixed locus of the action of $g_3.$} In that case 
$$ \Fix(g_3)=\{(\tau, \tau, x)\}_{\tau\in E[2],\; x\in E}\simeq \bigcup_{i\in \{1,2,3,4\}}E^{(i)} - \textup{four copies of $E$}$$ and $C(g_3)=\left\langle g_3\right\rangle$. Therefore $$H^{0,0}(\Fix(g_{3}))^{C(g_3)}\simeq  \CC^{4}\quad \textup{and}\quad H^{1,0}(\Fix(g_{3}))^{C(g_3)}\simeq \CC^{4}.$$

\textbf{The fixed locus of the action of $g_4.$} In that case 
$$\Fix(g_4)=\bigcup_{\tau_1\in E[2]}\bigcup_{\tau_2\in E[2]}\{(\tau_1, \tau_2, x)\}_{x\in E}\simeq \bigcup_{i,j\in \{1,2,3,4\}} E^{(i,j)} - \textup{sixteen copies of $E$}.$$
The action of $$C(g_4)=\left\langle g_4,\;
						\left(\begin{array}{rrr}-1 & 0 & 0 \\ 0 & 1 & 0 \\ 0 & 0 & -1 \end{array}\right), \left(\begin{array}{rrr}0 & 1 & 0 \\ -1 & 0 & 0 \\ 0 & 0 & 1 \end{array}\right) \right\rangle \simeq D_8 $$ on $H^{0,0}(\Fix(g_4))$ has the following matrix \newcommand\blockmat{%
		\begin{array}{@{\,}rr@{\,} }
			 0 & 1^{\mathstrut} \\ 1 & 0 \\ 
	\end{array}}
	$$(\mathds{1}_{})^{4}\oplus \left( \blockmat \right)^{6}$$
while the action of $C(g_4)$ on $H^{1,0}(\Fix(g_4))$ has the matrix $-\mathds{1}_{16}$. Therefore $$H^{0,0}(\Fix(g_{4}))^{C(g_3)}\simeq  \CC^{10}\quad \textup{and}\quad H^{1,0}(\Fix(g_{4}))^{C(g_3)}\simeq (0).$$

In total, the Calabi-Yau manifold $X$ has the following Hodge numbers:
	$$h^{1,1}(\kum_{3}(E,G_{24.1}))=1+4+1+4+10=20\quad \textup{and}\quad h^{2,1}(\kum_{3}(E,G_{24.1}))=1+0+1+4+0=6, $$ hence the Poincar{\'e} polynomial of $\kum_{3}(E,G_{24.1})$ equals $$t^6+20t^4+14t^3+20t^2+1, $$ which agrees with the result of \cite{Donten}. 
	
\subsection{Zeta function} 

Let $q$ be a power of a prime $p\neq 2$ such that the curve $E$ admits a good reduction $E_q$ over $\FF_q$. In the case of groups of order 6, 12, 24 we also assume that $p\neq 3$. These conditions imply that the Calabi-Yau 3-fold $X=\kum_{3}(E_q,G)$ has a good reduction $X_{q}$ over $\FF_q$.

The zeta function of $E_q$ is equal to $$\frac{1-a_qT+qT^2}{(1-T)(1-qT)},$$ where  $$a_q=q+1-\# E_q(\FF_q)\quad \textup{and}\quad |a_q|\leq 2\sqrt{q}.$$ Moreover trinomial $1-a_qT+qT^2$ has two roots $\alpha_{q}$ and $\bar{\alpha_{q}}$ which are algebraic integers such that $|\alpha_q|=|\bar{\alpha_q}|=q^{\frac{3}{2}}.$

By \cite{Rose} the orbifold zeta function is given by the product of zeta functions associated to fixed loci of actions of elements in $G_{24.1}.$ 
The orbifold zeta function associated to $H^{*}(E_q^{3})^{G_{24.1}}$ is a factor of zeta function of $E_q^3:$

$$\frac{\left(1-a_qT +qT^{2}\right)^{3}\left(1-(a_q^{3}-3 a_q q)T+q^3 T^2\right)\left(1-a_q qT +q^3T^{2}\right)^{9}\left(1-a_qq^2T +q^5T^{2}\right)^{3}}{\left(1-T \right)\left(1-qT\right)^{9} \left(1-(a_q^2-2q)T+ q^{2}T^{2}\right)^{3} \left(1-(a_q^2q-2q^2)T+ q^{4}T^{2}\right)^{3}\left(1-q^{2}T\right)^{9}\left(1-q^{3}T\right)}.$$


\textbf{The fixed locus of the action of $g_0=\id.$} The contribution to the zeta function coming from $H^{0}(X_q)$ and $H^{6}(X_q)$ obviously equals $$\frac{1}{1-T}\quad \textup{and}\quad \frac{1}{1-q^3T},$$ respectively. 

By the orbifold formula (cf. Thm. \ref{TY3}) the Kummer 3-fold $\kum_{3}(E,G)$ is Calabi-Yau exactly for $16$ non-cyclic groups from \cite{tahara} (see Table 2 in \cite{Donten} for a complete list of subgroups and corresponding Poincar{\'e} polynomials). 

\begin{Lem} Let $E$ be an elliptic curve with a good reduction $E_{q}$ over $\FF_q$. Then for any group $G$ (such that $\kum_{3}(E,G)$ is a Calabi-Yau 3-fold) the group $H^2(E_q^3)^{G}$ is generated by divisors defined over $\FF_q.$

Consequently, the orbifold Frobenius polynomials for $H^{2}(X_q)$ and $H^{4}(X_q)$ equal $$\frac{1}{(1-qT)^{\dim H^{1,1}(E^3)^G}}\quad \textup{and}\quad \frac{1}{(1-q^2T)^{\dim H^{1,1}(E^3)^G}},$$ respectively. \end{Lem}

\begin{proof} The vector space $H^{1,1}(E_q^3)^{G}$ is generated by the classes of the following invariant divisors:

\begin{minipage}{\linewidth}
      \begin{minipage}{0.45\linewidth}      
\begin{itemize}
    \item[\hyperlink{24.1}{$G_{24.1}\colon$}] 
	$\{(Z_1, Z_2, Z_3)\in E_q^{3}\colon Z_{1}=0\}$, \\
	$\{(Z_1, Z_2, Z_3)\in E_q^{3}\colon Z_{2}=0\}$,\\
	$\{(Z_1, Z_2, Z_3)\in E_q^{3}\colon Z_{3}=0\}$ \\
	
	\item[\hyperlink{24.2}{$G_{24.2}\colon$}]
	$\{(Z_1, Z_2, Z_3)\in E_q^{3}\colon Z_1+Z_2+Z_3=0\}$, \\
	$\{(Z_1, Z_2, Z_3)\in E_q^{3}\colon Z_1+Z_2=0\}$,\\
	$\{(Z_1, Z_2, Z_3)\in E_q^{3}\colon Z_1+Z_3=0\}$, \\
	$\{(Z_1, Z_2, Z_3)\in E_q^{3}\colon Z_{1}=0\}$, \\
	$\{(Z_1, Z_2, Z_3)\in E_q^{3}\colon Z_{2}=0\}$,\\
	$\{(Z_1, Z_2, Z_3)\in E_q^{3}\colon Z_{3}=0\}$\\
	
	\item[\hyperlink{24.3}{$G_{24.3}\colon$}] 
	$\{(Z_1, Z_2, Z_3)\in E_q^{3}\colon Z_1+Z_2+Z_3=0\}$,\\
	$\{(Z_1, Z_2, Z_3)\in E_q^{3}\colon Z_1+Z_2=0\}$,\\
	$\{(Z_1, Z_2, Z_3)\in E_q^{3}\colon Z_1+Z_3=0\}$, \\
	$\{(Z_1, Z_2, Z_3)\in E_q^{3}\colon Z_{1}=0\}$, \\
	$\{(Z_1, Z_2, Z_3)\in E_q^{3}\colon Z_{2}=0\}$,\\
	$\{(Z_1, Z_2, Z_3)\in E_q^{3}\colon Z_{3}=0\}$\\
	
	\item[\hyperlink{12.1}{$G_{12.1}\colon$}]
	$\{(Z_1, Z_2, Z_3)\in E_q^{3}\colon Z_2+Z_3=0\}$, \\
	$\{(Z_1, Z_2, Z_3)\in E_q^{3}\colon Z_2=0\}$,\\
	$\{(Z_1, Z_2, Z_3)\in E_q^{3}\colon Z_3=0\}$, \\
	$\{(Z_1, Z_2, Z_3)\in E_q^{3}\colon Z_1=0\}$\\
    
    \item[\hyperlink{12.2}{$G_{12.2}\colon$}]
	$\{(Z_1, Z_2, Z_3)\in E_q^{3}\colon Z_{1}=0\}$, \\
	$\{(Z_1, Z_2, Z_3)\in E_q^{3}\colon Z_{2}=0\}$,\\
	$\{(Z_1, Z_2, Z_3)\in E_q^{3}\colon Z_{3}=0\}$\\
	
	\item[\hyperlink{12.3}{$G_{12.3}\colon$}]
	$\{(Z_1, Z_2, Z_3)\in E_q^{3}\colon Z_1-Z_2=0\}$, \\
	$\{(Z_1, Z_2, Z_3)\in E_q^{3}\colon Z_2-Z_3=0\}$,\\
	$\{(Z_1, Z_2, Z_3)\in E_q^{3}\colon Z_3-Z_1=0\}$, \\
	$\{(Z_1, Z_2, Z_3)\in E_q^{3}\colon Z_{1}=0\}$, \\
	$\{(Z_1, Z_2, Z_3)\in E_q^{3}\colon Z_{2}=0\}$,\\
	$\{(Z_1, Z_2, Z_3)\in E_q^{3}\colon Z_{3}=0\}$\\
	
	\item[\hyperlink{12.4}{$G_{12.4}\colon$}]
	$\{(Z_1, Z_2, Z_3)\in E_q^{3}\colon Z_1+Z_2=0\}$, \\
	$\{(Z_1, Z_2, Z_3)\in E_q^{3}\colon Z_1+Z_3=0\}$,\\
	$\{(Z_1, Z_2, Z_3)\in E_q^{3}\colon Z_2+Z_3=0\}$, \\
	$\{(Z_1, Z_2, Z_3)\in E_q^{3}\colon Z_1+Z_2+Z_3=0\}$,\\
	$\{(Z_1, Z_2, Z_3)\in E_q^{3}\colon Z_{1}=0\}$, \\
	$\{(Z_1, Z_2, Z_3)\in E_q^{3}\colon Z_{2}=0\}$,\\
	$\{(Z_1, Z_2, Z_3)\in E_q^{3}\colon Z_{3}=0\}$\\
\end{itemize}  
 \end{minipage}   
 \hspace{11mm}
 \begin{minipage}{0.45\linewidth}
\begin{itemize}
    \item[\hyperlink{8.1}{$G_{8.1}\colon$}]
	$\{(Z_1, Z_2, Z_3)\in E_q^{3}\colon Z_{1}=0\}$, \\
	$\{(Z_1, Z_2, Z_3)\in E_q^{3}\colon Z_{2}=0\}$,\\
	$\{(Z_1, Z_2, Z_3)\in E_q^{3}\colon Z_{3}=0\}$\\
	
	\item[\hyperlink{8.2}{$G_{8.2}\colon$}]
	$\{(Z_1, Z_2, Z_3)\in E_q^{3}\colon Z_1+Z_2+Z_3=0\}$, \\
	$\{(Z_1, Z_2, Z_3)\in E_q^{3}\colon Z_1+Z_2=0\}$,\\
	$\{(Z_1, Z_2, Z_3)\in E_q^{3}\colon Z_1+Z_3=0\}$, \\
	$\{(Z_1, Z_2, Z_3)\in E_q^{3}\colon Z_{1}=0\}$, \\
	$\{(Z_1, Z_2, Z_3)\in E_q^{3}\colon Z_{2}=0\}$,\\
	$\{(Z_1, Z_2, Z_3)\in E_q^{3}\colon Z_{3}=0\}$\\

	\item[\hyperlink{6.1}{$G_{6.1}\colon$}]
	$\{(Z_1, Z_2, Z_3)\in E_q^{3}\colon Z_3-Z_2=0\}$, \\
	$\{(Z_1, Z_2, Z_3)\in E_q^{3}\colon Z_{1}=0\}$, \\
	$\{(Z_1, Z_2, Z_3)\in E_q^{3}\colon Z_{2}=0\}$,\\
	$\{(Z_1, Z_2, Z_3)\in E_q^{3}\colon Z_{3}=0\}$\\
	
	\item[\hyperlink{6.2}{$G_{6.2}\colon$}] 
	$\{(Z_1, Z_2, Z_3)\in E_q^{3}\colon Z_3-Z_2=0\}$, \\
	$\{(Z_1, Z_2, Z_3)\in E_q^{3}\colon Z_{1}=0\}$, \\
	$\{(Z_1, Z_2, Z_3)\in E_q^{3}\colon Z_{2}=0\}$,\\
	$\{(Z_1, Z_2, Z_3)\in E_q^{3}\colon Z_{3}=0\}$\\
	
	\item[\hyperlink{6.3}{$G_{6.3}\colon$}] 
	$\{(Z_1, Z_2, Z_3)\in E_q^{3}\colon Z_1+Z_2+Z_3=0\}$\\ 
		
	\item[\hyperlink{4.1}{$G_{4.1}\colon$}] 
	$\{(Z_1, Z_2, Z_3)\in E_q^{3}\colon Z_{1}=0\}$, \\
	$\{(Z_1, Z_2, Z_3)\in E_q^{3}\colon Z_{2}=0\}$,\\
	$\{(Z_1, Z_2, Z_3)\in E_q^{3}\colon Z_{3}=0\}$\\
	
	\item[\hyperlink{4.2}{$G_{4.2}\colon$}] 
	$\{(Z_1, Z_2, Z_3)\in E_q^{3}\colon Z_2+Z_3=0\}$, \\
	$\{(Z_1, Z_2, Z_3)\in E_q^{3}\colon Z_{1}=0\}$, \\
	$\{(Z_1, Z_2, Z_3)\in E_q^{3}\colon Z_{2}=0\}$,\\
	$\{(Z_1, Z_2, Z_3)\in E_q^{3}\colon Z_{3}=0\}$\\
	
	\item[\hyperlink{4.3}{$G_{4.3}\colon$}]
	$\{(Z_1, Z_2, Z_3)\in E_q^{3}\colon Z_{1}=0\}$, \\
	$\{(Z_1, Z_2, Z_3)\in E_q^{3}\colon Z_{2}=0\}$,\\
	$\{(Z_1, Z_2, Z_3)\in E_q^{3}\colon Z_{3}=0\}$\\
	
	\item[\hyperlink{4.4}{$G_{4.4}\colon$}]
	$\{(Z_1, Z_2, Z_3)\in E_q^{3}\colon Z_1-Z_2=0\}$, \\
	$\{(Z_1, Z_2, Z_3)\in E_q^{3}\colon Z_1+Z_3=0\}$,\\
	$\{(Z_1, Z_2, Z_3)\in E_q^{3}\colon Z_2-Z_3=0\}$, \\
	$\{(Z_1, Z_2, Z_3)\in E_q^{3}\colon Z_{2}=0\}$,\\
	$\{(Z_1, Z_2, Z_3)\in E_q^{3}\colon Z_{3}=0\}$\\
\end{itemize}           
     \end{minipage}
  \end{minipage}

\end{proof}

\begin{Lem} The orbifold Frobenius polynomial for $H_{\et}^{3}(E^3_q, \QQ_{\ell})^{G}$ equals 
$$\left(1-(a_q^{3}-3 a_q q)T+q^3 T^2\right)\left(1-a_q qT +q^3T^{2}\right)^{\dim H^{2,1}(E^3, \CC)^G}.$$\end{Lem}

\begin{proof} The lemma is obvious when $a_q=0,$ so we can consider only ordinary case. Let $\omega, \eta\in H_{\textup{DR}}^{1}(E_q)$ such that $$\Frob_q^{*}\omega =\alpha_q \omega\quad \textup{and} \quad \Frob_q^{*}\eta =\bar{\alpha_q} \eta.$$
Denote by $\diff z_{i}=\pi_{i}^{*} (\diff z),$ $\omega_i=\pi_{i}^{*} (\omega),$ $\eta_i=\pi_{i}^{*} (\eta)$ where $\pi_{i}\colon E^{3}\to E$ is projection on $i$-th coordinate $(i=1,2,3)$.
As $\omega$ and $\eta$ are linearly independent it follows $$\diff z=k \omega +\ell \eta\quad \textup{and}\quad \diff \bar{z}=m \omega +n \eta,$$ for some $k, \ell, m, n \in \CC$. Now 
$$\diff z_{1}\wedge \diff z_{2}\wedge \diff z_{3}=k^3 \omega_1\wedge \omega_2\wedge \omega_3 +\ell^3 \eta_1\wedge \eta_2\wedge \eta_3+k^{2}\ell \xi +k^{}\ell^{2}\psi, $$ where $$\Frob_{q}^{*}(\xi)=q\alpha_q\xi, \quad \quad \Frob_q^{*}(\omega_{1}\wedge \omega_{2}\wedge \omega_{3})=\alpha_{q}^3\omega_{1}\wedge \omega_{2}\wedge \omega_{3}$$ and
$$\quad \Frob_{q}^{*}(\psi)=q\bar{\alpha_q} \psi,\quad \quad\Frob_q^{*}(\eta_{1}\wedge \eta_{2}\wedge \eta_{3})=\bar{\alpha_q}^3\eta_{1}\wedge \eta_{2}\wedge \eta_{3}.$$ 
If all eigenvalues of $\Frob_q^{*}$ are equal $q\alpha_{q}$ and $q\bar{\alpha_{q}},$ then $k^{3}=\ell^{3}=0$ as we consider the case $a_q \neq 0$. Consequently $\diff z=0$ which is impossible in the ordinary case. 
\end{proof}
From the above lemmas we immediately get the following corollary
\begin{Cor}
The orbifold zeta function for $H_{\et}^{3}(E^3_q, \QQ_{\ell})^{G}$ equals
$$\frac{\left(1-(a_q^{3}-3 a_q q)T+q^3 T^2\right)\left(1-a_q qT +q^3T^{2}\right)^{\dim H^{2,1}(E^3, \CC)^G}}{(1-T)(1-qT)^{\dim H^{1,1}(E^3, \CC)^G}(1-q^2T)^{\dim H^{1,1}(E^3, \CC)^G}(1-q^3T)}.$$ \end{Cor}

For our main case \hyperlink{24.1}{(24.1)} we have computed  $\dim H^{1,1}(E^3, \CC)^{G_{24.1}}=\dim H^{2,1}(E^3, \CC)^{G_{24.1}}=1,$ so the orbifold zeta function $H_{\et}^{3}(E^3_q, \QQ_{\ell})^{G}$ equals 

\begin{equation} \label{r0} \frac{\left(1-(a_q^{3}-3 a_q q)T+q^3 T^2\right)\left(1-a_q qT +q^3T^{2}\right)^{\dim H^{2,1}(E^3, \CC)^G}}{(1-T)(1-qT)^{\dim H^{1,1}(E^3, \CC)^G}(1-q^2T)^{\dim H^{1,1}(E^3, \CC)^G}(1-q^3T)}.\end{equation}

\textbf{The fixed locus of the action of $g_1=\left(\begin{array}{rrr}-1 & 0 & 0\\0 & 0 & -1\\ 0 & -1 & 0 \end{array}\right)$}

Since $\Fix(g_{1})/C(g_1)\simeq \textup{four copies of $\PP^{1}$}$ indexed by 2-torsion points $E[2]$ the zeta function of $H^{*}(\Fix(g_{1})/C(g_1))$
depends on the behaviour of torsion subgroups $E[2]$ under the reduction modulo $q$. Let $E_q[2]=\{a,b,c,d\}$ and denote by $P_{(a_{1}, a_{2}, \ldots, a_{k})}$ the characteristic polynomial of the following matrix:

$$\underbrace{\begin{pmatrix}
			0 & 0 & 0 &\ldots & 1\\
			1 & 0 & 0 & \ldots & 0\\
			0 & 1 & 0 & \ldots & 0\\
			\vdots & \vdots & \vdots &  \vdots & \vdots\\
			0 &0 & \ldots & 1 & 0 \\
		\end{pmatrix}}_{a_{1}}\oplus \underbrace{\begin{pmatrix}
			0 & 0 & 0 &\ldots & 1\\
			1 & 0 & 0 & \ldots & 0\\
			0 & 1 & 0 & \ldots & 0\\
			\vdots & \vdots & \vdots &  \vdots & \vdots\\
			0 &0 & \ldots & 1 & 0 \\
		\end{pmatrix}}_{a_{2}} \oplus \cdots \oplus \underbrace{\begin{pmatrix}
			0 & 0 & 0 &\ldots & 1\\
			1 & 0 & 0 & \ldots & 0\\
			0 & 1 & 0 & \ldots & 0\\
			\vdots & \vdots & \vdots &  \vdots & \vdots\\
			0 &0 & \ldots & 1 & 0 \\
		\end{pmatrix}}_{a_{k}} $$
		
\vspace{4mm}
 
We have the following three cases:
\begin{itemize}
	\item[\hypertarget{H4}{$(1,1,1,1)$}] All points $a,$ $b,$ $c,$ $d$ are defined over $\FF_q$. Then the zeta function is equal to $$P_{(1,1,1,1)}(T)=(1-T)^{4}.$$ 
	\item[\hypertarget{H2}{$(1,1,2)$}] Points $a,$ $b,$ $c+d$ are defined over $\FF_q$ but $c,$ $d$ are not. The action of the Frobenius has the following linearization $$
	{
		\begin{pmatrix}
			1 & 0 & 0 & 0\\
			0 & 1 & 0 & 0\\
			0 & 0 & 0 & 1\\
			0 & 0 & 1 & 0\\
		\end{pmatrix}
	}
	$$
	and therefore the zeta function equals  $$P_{(1,1,2)}(T)=(1-T)^{3}(1+T).$$
	\item[\hypertarget{H1}{$(1,3)$}] $a$ is defined over $\FF_{q}$ and $b,$ $ c,$ $d$ not. The action of the Frobenius has the following linearization $$
	{
		\begin{pmatrix}
			1 & 0 & 0 & 0\\
			0 & 0 & 0 & 1\\
			0 & 1 & 0 & 0\\
			0 & 0 & 1 & 0\\
		\end{pmatrix}
	}
	$$
	and therefore the zeta function equals $$P_{(1,3)}(T)=(1-T)^{2}(1+T+T^2).$$
\end{itemize} 

From section \ref{zeta} it follows that in the case $(\lambda)$ we have the following formula
\begin{equation}\label{r1} Z_{q}\left(H^{*}(\Fix(g_{1})/C(g_1))\right)=Z_{q}(H^{*}(\PP^1))(qT)\otimes P_{(\lambda)}(T)=\frac{1}{\left((1-qT)(1-q^2T)\right)\otimes P_{(\lambda)}(T)},\end{equation} for $(\lambda)\in\{(1,1,1,1), (1,1,2), (1,3)\}.$

\textbf{The fixed locus of the action of $g_2=\left(\begin{array}{rrr}0 & -1 & 0\\0 & 0 & -1\\ 1 & 0 & 0 \end{array}\right)$}

Since $\Fix(g_{2})/C(g_2)\simeq E$ it follows that 
\begin{equation} \label{r2} Z_{q}\left(H^{*}(\Fix(g_{2})/C(g_2))\right)=Z_{q}(H^{*}(E))(qT)=\frac{1-a_qqT+q^3T^2}{(1-qT)(1-q^2T)}.\end{equation} 

\textbf{The fixed locus of the action of $g_3=\left(\begin{array}{rrr}0 & -1 & 0\\1 & 0 & 0\\ 0 & 0 & 1 \end{array}\right)$}
Since $C(g_3)$ is generated by $g_3$ then 
\begin{equation}\label{r3} Z_{q}\left(H^{*}(\Fix(g_{3})/C(g_3))\right)=Z_{q}(H^{*}(E))(qT)\otimes P_{(\lambda)}(T)=\frac{1-a_qqT+q^3T^2}{(1-qT)(1-q^2T)}\otimes P_{(\lambda)}(T).\end{equation}

\textbf{The fixed locus of the action of $g_4=\left(\begin{array}{rrr}-1 & 0 & 0\\0 & -1 & 0\\ 0 & 0 & 1 \end{array}\right)$}

In the case of $\hyperlink{H4}{(1,1,1,1)}$ the action has the matrix $\mathds{1}_{10}$ with characteristic polynomial equal to
$$P_{(1,1,1,1,1,1,1,1,1,1)}(T)=(1-T)^{10}.$$

For \hyperlink{H2}{(1,1,2)} the Frobenius morphism has the following linearization: $$(\mathds{1}_{})^{4}\oplus \left( \blockmat \right)^{3}$$
with characteristic polynomial equal to $$P_{(1,1,1,1,2,2,2)}(T)=(1-T)^{4}(1-T^2)^{3}.$$

Finally, for \hyperlink{H1}{(1,3)} the Frobenius morphism has the following linearization:

\newcommand\blockmatt{%
		\begin{array}{@{\,}rrr@{\,} }
			 0 & 1 & 0 \\ 0 & 0 & 1\\ 1 & 0 & 0
	\end{array}}
$$(\mathds{1}_{})^{4}\oplus \left( \blockmatt \right)^{2}$$	
with characteristic polynomial $$P_{(1,1,1,1,3,3)}(T)=(1-T)^6(1+T+T^2)^2.$$
	
Therefore in the case $(\eta)$ we get
\begin{equation}\label{r4} Z_{q}\left(H^{*}(\Fix(g_{4})/C(g_4))\right)=Z_{q}(H^{*}(\PP^1))(qT)\otimes P_{(\eta)}(T)=\frac{1}{\left((1-qT)(1-q^2T)\right)\otimes P_{(\eta)}(T)},\end{equation} for $\eta\in \{(1,1,1,1,1,1,1,1,1,1), (1,1,1,1,2,2,2), (1,1,1,1,3,3)\}.$ 

Multiplying rational functions (\ref{r0}), (\ref{r1}), (\ref{r2}), (\ref{r3}), (\ref{r4}) we obtain the following
 
\begin{Thm} The zeta function of $\kum_{3}(E, G_{24.1})$ equals 

\begin{table}[H]
	\begin{center}
		\begin{varwidth}{\textheight}
			\resizebox{0.55\textheight}{!}{
				\begin{tabular}{c||c}
				\textup{Case:}	 &
					\begin{tabular}{@{}l@{}} \setlength{\fboxrule}{1pt}\fcolorbox{white}{white}{$\displaystyle{Z_{q}\left(\kum_{3}(E, G_{24.1})\right)}$}
					\end{tabular} \\
					\hhline{=::=:}
					$(1,1,1,1)$ &  
					\begin{tabular}{@{}l@{}}
						\setlength{\fboxrule}{1pt}\fcolorbox{white}{white}{$\displaystyle{\frac{\left(1-(a_q^{3}-3 a_q q)T+q^3 T^2\right)\left(1-a_q qT +q^3T^{2}\right)^{6}}{(1-T)(1-qT)^{20}(1-q^2T)^{20}(1-q^3T)}}$} 
					\end{tabular} \\
					\hhline{=::=:}
					$(1,1,2)$ &
					\begin{tabular}{@{}l@{}} 
						\setlength{\fboxrule}{1pt}\fcolorbox{white}{white}{$\displaystyle{\frac{\left(1-(a_q^{3}-3 a_q q)T+q^3 T^2\right)\left(1-a_q qT +q^3T^{2}\right)^{5}\left(1+a_q qT +q^3T^{2}\right)^{}}{(1-T)(1-qT)^{12}(1+qT)^5(1-a_{q}qT+q^{3}T^2)^3(1+q^2T)^5(1-q^2T)^{12}(1-q^3T)}}$} 
					\end{tabular}\\
				   \hhline{=::=:}
				$(1,3)$ &
				\begin{tabular}{@{}l@{}} 
					\setlength{\fboxrule}{1pt}\fcolorbox{white}{white}{$\displaystyle{\frac{\left(1-(a_q^{3}-3 a_q q)T+q^3 T^2\right)\left(1-a_q qT +q^3T^{2}\right)^{4}\left(1+a_q qT -(q^3-a_q^2 q^2) T^2+a_q q^4 T^3 +q^6 T^4\right)^{}}{(1-T)(1-qT)^{12}(1+qT+q^{2}T^2)^4(1+q^2T+q^4T^2)^4(1-q^2T)^{12}(1-q^3T)}}$} 
				\end{tabular}\\
			\end{tabular}}	
		\end{varwidth}
		\vspace{2mm}
	\end{center}
\end{table}

\end{Thm}

\section{Examples}
\label{results}

In the computation of zeta function we excluded powers of $2$ and $3$ because $2$ is a prime of bad reduction of all considered Kummer manifolds while $3$ is a prime of bad reduction of $\kum_{3}(E,G)$ where $G$ is a group of order $6,$ $12,$ $24.$

In this section we collect result of computations for all remaining groups. The Frobenius morphism acts on $E[2]$ and $E[3]$ by a permutation and the shape of the zeta function depends on the cycle type of this permutation. In particular, if all elements in $E[2]$ and $E[3]$ are $\FF_q$ rational then this action is trivial and hence the zeta function of $X_q$ is given by the default (simplest) formula corresponding to the first row in the table.
\subsection{Order 24}

\textup{     }
\\
\linia
\vspace{1mm}
\begin{center}$G_{24.2}\colon\;\; $$\left\langle\left(\begin{array}{rrr} 0 & -1 & 0 \\ 1 & 1 & 1 \\ -1 & 0 & 0 \end{array}\right), \left(\begin{array}{rrr} -1 & -1 & 0 \\ 0 & 1 & 0 \\ 0 & 0 & -1 \end{array}\right) \right\rangle\simeq S_{4}$\end{center} 
\vspace{1mm}
\linia
\begin{itemize}
\item[] \textbf{Hodge numbers: } $$h^{1,1}(\kum_{3}(E, G_{24.2}))=11\quad \textup{and} \quad h^{2,1}(\kum_{3}(E, G_{24.2}))=3$$
\item[] \textbf{Zeta function: } 

\begin{table}[H]
	\begin{center}
		\begin{varwidth}{\textheight}
			\resizebox{0.55\textheight}{!}{
				\begin{tabular}{c||c}
				\textrm{Case:}	 &
					\begin{tabular}{@{}l@{}} \setlength{\fboxrule}{1pt}\fcolorbox{white}{white}{$\displaystyle{Z_{q}\left(\kum_{3}(E, G_{24.2})\right)}$}
					\end{tabular} \\
					\hhline{=::=:}
					$(1,1,1,1)$ &  
					\begin{tabular}{@{}l@{}}
						\setlength{\fboxrule}{1pt}\fcolorbox{white}{white}{$\displaystyle{\frac{\left(1-(a_q^{3}-3 a_q q)T+q^3 T^2\right)\left(1-a_q qT +q^3T^{2}\right)^{3}}{(1-T)(1-qT)^{11}(1-q^2T)^{11}(1-q^3T)}}$} 
					\end{tabular} \\
					\hhline{=::=:}
					$(1,1,2)$ &
					\begin{tabular}{@{}l@{}} 
						\setlength{\fboxrule}{1pt}\fcolorbox{white}{white}{$\displaystyle{\frac{\left(1-(a_q^{3}-3 a_q q)T+q^3 T^2\right)\left(1-a_q qT +q^3T^{2}\right)^{3}}{(1-T)(1-qT)^{10}(1+qT)(1+q^2T)(1-q^2T)^{10}(1-q^3T)}}$} 
					\end{tabular}\\
				   \hhline{=::=:}
				$(1,3)$ &
				\begin{tabular}{@{}l@{}} 
					\setlength{\fboxrule}{1pt}\fcolorbox{white}{white}{$\displaystyle{\frac{\left(1-(a_q^{3}-3 a_q q)T+q^3 T^2\right)\left(1-a_q qT +q^3T^{2}\right)^{3}}{(1-T)(1-qT)^{7}(1+qT+q^{2}T^2)^2(1+q^2T+q^4T^2)^2(1-q^2T)^{7}(1-q^3T)}}$} 
				\end{tabular}\\
			\end{tabular}}	
		\end{varwidth}
		\vspace{2mm}
	\end{center}
\end{table}
\end{itemize}

\vspace{-4mm}
\textup{     }
\\
\linia
\vspace{1mm}
\begin{center}$G_{24.3}\colon\;\;\left\langle\left(\begin{array}{rrr} 1 & 1 & 0 \\ -2 & -1 & -1 \\ 0 & 0 & 1 \end{array}\right), \left(\begin{array}{rrr} -1 & -1 & -1 \\ 0 & 0 & 1 \\ 0 & 1 & 0 \end{array}\right) \right\rangle\simeq S_{4}$ \end{center} 
\vspace{1mm}
\linia
\begin{itemize}
\item[] \textbf{Hodge numbers: } $$h^{1,1}(\kum_{3}(E, G_{24.3}))=11\quad \textup{and} \quad h^{2,1}(\kum_{3}(E, G_{24.3}))=3$$ 
\item[] \textbf{Zeta function: } 

\begin{table}[H]
	\begin{center}
		\begin{varwidth}{\textheight}
			\resizebox{0.55\textheight}{!}{
				\begin{tabular}{c||c}
				\textrm{Case:}	 &
					\begin{tabular}{@{}l@{}} \setlength{\fboxrule}{1pt}\fcolorbox{white}{white}{$\displaystyle{Z_{q}\left(\kum_{3}(E, G_{24.3})\right)}$}
					\end{tabular} \\
					\hhline{=::=:}
					$(1,1,1,1)$ &  
					\begin{tabular}{@{}l@{}}
						\setlength{\fboxrule}{1pt}\fcolorbox{white}{white}{$\displaystyle{\frac{\left(1-(a_q^{3}-3 a_q q)T+q^3 T^2\right)\left(1-a_q qT +q^3T^{2}\right)^{3}}{(1-T)(1-qT)^{11}(1-q^2T)^{11}(1-q^3T)}}$} 
					\end{tabular} \\
					\hhline{=::=:}
					$(1,1,2)$ &
					\begin{tabular}{@{}l@{}} 
						\setlength{\fboxrule}{1pt}\fcolorbox{white}{white}{$\displaystyle{\frac{\left(1-(a_q^{3}-3 a_q q)T+q^3 T^2\right)\left(1-a_q qT +q^3T^{2}\right)^{3}}{(1-T)(1-qT)^{10}(1+qT)(1+q^2T)(1-q^2T)^{10}(1-q^3T)}}$} 
					\end{tabular}\\
				   \hhline{=::=:}
				$(1,3)$ &
				\begin{tabular}{@{}l@{}} 
					\setlength{\fboxrule}{1pt}\fcolorbox{white}{white}{$\displaystyle{\frac{\left(1-(a_q^{3}-3 a_q q)T+q^3 T^2\right)\left(1-a_q qT +q^3T^{2}\right)^{3}}{(1-T)(1-qT)^{7}(1+qT+q^{2}T^2)^2(1+q^2T+q^4T^2)^2(1-q^2T)^{7}(1-q^3T)}}$} 
				\end{tabular}\\
			\end{tabular}}	
		\end{varwidth}
		\vspace{2mm}
	\end{center}
\end{table}
\end{itemize}
\newpage
\subsection{Order 12}
\textup{     }
\\
\linia
\vspace{1mm}
\begin{center}$G_{12.1}\colon\;\; \left\langle\left(\begin{array}{rrr} 1 & 0 & 0 \\ 0 & 0 & -1 \\ 0 & 1 & 1 \end{array}\right), \left(\begin{array}{rrr} -1 & 0 & 0 \\ 0 & 0 & 1 \\ 0 & 1 & 0 \end{array}\right) \right\rangle\simeq D_{6}$ \end{center} 
\vspace{1mm}
\linia
\begin{itemize}
\item[] \textbf{Hodge numbers: } $$h^{1,1}(\kum_{3}(E, G_{12.1}))=21\quad \textup{and} \quad h^{2,1}(\kum_{3}(E, G_{12.1}))=9$$
\item[] \textbf{Zeta function: } 
We get 15 possibilities for the zeta function of $\kum_{3}(E, G_{12.1})$:

$$\frac{\left(1-(a_q^{3}-3 a_q q)T+q^3 T^2\right)\left(1-a_q qT +q^3T^{2}\right)^{4}}{(1-T)(1-qT)^{4}\left(((1-qT)(1-q^2T))\otimes P_{(\eta)}\right)^{3}(1-q^2T)^{4}(1-q^3T)}\cdot \left(\left(\frac{1-a_q qT +q^3T^{2}}{(1-qT)(1-q^2T)}\right)\otimes P_{(\lambda)}\right),$$ where
$$((\lambda), (\eta))\in \left\{(1,1,1,1,1), (1,1,1,2), (1,2,2), (1,1,3), (1,4)\right\}\times \{(1,1,1,1), (1,1,2), (1,3)\}$$
\end{itemize}

\begin{table}[H]
	\begin{center}
		\begin{varwidth}{\textheight}
			\resizebox{0.62\textheight}{!}{
				\begin{tabular}{c||c}
				\begin{tabular}{@{}l@{}} \setlength{\fboxrule}{1pt}\fcolorbox{white}{white}{\textrm{Case:}}
					\end{tabular}	 &
					\begin{tabular}{@{}l@{}} \setlength{\fboxrule}{2pt}\fcolorbox{white}{white}{$\displaystyle{Z_{q}\left(\kum_{3}(E, G_{12.1})\right)}$}
					\end{tabular} \\
					\hhline{=::=:}
					\begin{tabular}{c@{}c@{}} \setlength{\fboxrule}{2pt}\fcolorbox{white}{white}{$(1,1,1,1,1)$}\\
					    \setlength{\fboxrule}{2pt}\fcolorbox{white}{white}{$(1,1,1,1)$}
					\end{tabular} &  
					\begin{tabular}{@{}l@{}}
						\setlength{\fboxrule}{2pt}\fcolorbox{white}{white}{$\displaystyle{\frac{\left(1-(a_q^{3}-3 a_q q)T+q^3 T^2\right)\left(1-a_q qT +q^3T^{2}\right)^{9}}{(1-T)(1-qT)^{21}(1-q^2T)^{21}(1-q^3T)}}$} 
					\end{tabular} \\
					\hhline{=::=:}
					\begin{tabular}{c@{}c@{}} \setlength{\fboxrule}{2pt}\fcolorbox{white}{white}{$(1,1,1,1,1)$}\\
					    \setlength{\fboxrule}{2pt}\fcolorbox{white}{white}{$(1,1,2)$}
					\end{tabular} &
					\begin{tabular}{@{}l@{}} 
						\setlength{\fboxrule}{2pt}\fcolorbox{white}{white}{$\displaystyle{\frac{\left(1-(a_q^{3}-3 a_q q)T+q^3 T^2\right)\left(1-a_q qT +q^3T^{2}\right)^{9}}{(1-T)(1-qT)^{18}(1+qT)^3(1+q^2T)^3(1-q^2T)^{18}(1-q^3T)}}$} 
					\end{tabular}\\
				   \hhline{=::=:}
				\begin{tabular}{c@{}c@{}} \setlength{\fboxrule}{2pt}\fcolorbox{white}{white}{$(1,1,1,1,1)$}\\
					    \setlength{\fboxrule}{2pt}\fcolorbox{white}{white}{$(1,3)$}
					\end{tabular} &
				\begin{tabular}{@{}l@{}} 
					\setlength{\fboxrule}{2pt}\fcolorbox{white}{white}{$\displaystyle{\frac{\left(1-(a_q^{3}-3 a_q q)T+q^3 T^2\right)\left(1-a_q qT +q^3T^{2}\right)^{9}}{(1-T)(1-qT)^{15}(1+qT+q^{2}T^2)^3(1+q^2T+q^4T^2)^3(1-q^2T)^{15}(1-q^3T)}}$} 
				\end{tabular}\\
				\hhline{=::=:}
				\begin{tabular}{c@{}c@{}} \setlength{\fboxrule}{2pt}\fcolorbox{white}{white}{$(1,1,1,2)$}\\
					    \setlength{\fboxrule}{2pt}\fcolorbox{white}{white}{$(1,1,1,1)$}
					\end{tabular} &  
					\begin{tabular}{@{}l@{}}
						\setlength{\fboxrule}{2pt}\fcolorbox{white}{white}{$\displaystyle{\frac{\left(1-(a_q^{3}-3 a_q q)T+q^3 T^2\right)\left(1+a_q qT +q^3T^{2}\right)^{5}\left(1-a_q qT +q^3T^{2}\right)^{4}}{(1-T)(1-qT)^{20}(1+qT)^{}(1+q^2T)^{}(1-q^2T)^{20}(1-q^3T)}}$} 
					\end{tabular} \\
					\hhline{=::=:}
					\begin{tabular}{c@{}c@{}} \setlength{\fboxrule}{2pt}\fcolorbox{white}{white}{$(1,1,1,2)$}\\
					    \setlength{\fboxrule}{2pt}\fcolorbox{white}{white}{$(1,1,2)$}
					\end{tabular} &
					\begin{tabular}{@{}l@{}} 
						\setlength{\fboxrule}{2pt}\fcolorbox{white}{white}{$\displaystyle{\frac{\left(1-(a_q^{3}-3 a_q q)T+q^3 T^2\right)\left(1+a_q qT +q^3T^{2}\right)^{5}\left(1-a_q qT +q^3T^{2}\right)^{4}}{(1-T)(1-qT)^{17}(1+qT)^{4}(1+q^2T)^{4}(1-q^2T)^{17}(1-q^3T)}}$} 
					\end{tabular}\\
				   \hhline{=::=:}
				\begin{tabular}{c@{}c@{}} \setlength{\fboxrule}{2pt}\fcolorbox{white}{white}{$(1,1,1,2)$}\\
					    \setlength{\fboxrule}{2pt}\fcolorbox{white}{white}{$(1,3)$}
					\end{tabular} &
				\begin{tabular}{@{}l@{}} 
					\setlength{\fboxrule}{2pt}\fcolorbox{white}{white}{$\displaystyle{\frac{\left(1-(a_q^{3}-3 a_q q)T+q^3 T^2\right)\left(1+a_q qT +q^3T^{2}\right)^{5}\left(1-a_q qT +q^3T^{2}\right)^{4}}{(1-T)(1-qT)^{14}(1+qT)^{}(1+qT+q^{2}T^2)^3(1+q^2T+q^4T^2)^3(1+q^2T)^{}(1-q^2T)^{14}(1-q^3T)}}$} 
				\end{tabular}\\
				\hhline{=::=:}
				\begin{tabular}{c@{}c@{}} \setlength{\fboxrule}{2pt}\fcolorbox{white}{white}{$(1,2,2)$}\\
					    \setlength{\fboxrule}{2pt}\fcolorbox{white}{white}{$(1,1,1,1)$}
					\end{tabular} &  
					\begin{tabular}{@{}l@{}}
						\setlength{\fboxrule}{2pt}\fcolorbox{white}{white}{$\displaystyle{\frac{\left(1-(a_q^{3}-3 a_q q)T+q^3 T^2\right)\left(1+a_q qT +q^3T^{2}\right)^{7}\left(1-a_q qT +q^3T^{2}\right)^{2}}{(1-T)(1-qT)^{19}(1+qT)^{2}(1+q^2T)^{2}(1-q^2T)^{19}(1-q^3T)}}$} 
					\end{tabular} \\


				\hhline{=::=:}
					\begin{tabular}{c@{}c@{}} \setlength{\fboxrule}{2pt}\fcolorbox{white}{white}{$(1,2,2)$}\\
					    \setlength{\fboxrule}{2pt}\fcolorbox{white}{white}{$(1,1,2)$}
					\end{tabular} &
					\begin{tabular}{@{}l@{}} 
						\setlength{\fboxrule}{2pt}\fcolorbox{white}{white}{$\displaystyle{\frac{\left(1-(a_q^{3}-3 a_q q)T+q^3 T^2\right)\left(1+a_q qT +q^3T^{2}\right)^{7}\left(1-a_q qT +q^3T^{2}\right)^{2}}{(1-T)(1-qT)^{16}(1+qT)^{5}(1+q^2T)^{5}(1-q^2T)^{16}(1-q^3T)}}$} 
					\end{tabular}\\
				   \hhline{=::=:}
				\begin{tabular}{c@{}c@{}} \setlength{\fboxrule}{2pt}\fcolorbox{white}{white}{$(1,2,2)$}\\
					    \setlength{\fboxrule}{2pt}\fcolorbox{white}{white}{$(1,3)$}
					\end{tabular} &
				\begin{tabular}{@{}l@{}} 
					\setlength{\fboxrule}{2pt}\fcolorbox{white}{white}{$\displaystyle{\frac{\left(1-(a_q^{3}-3 a_q q)T+q^3 T^2\right)\left(1+a_q qT +q^3T^{2}\right)^{7}\left(1-a_q qT +q^3T^{2}\right)^{2}}{(1-T)(1-qT)^{13}(1+qT)^{2}(1+qT+q^{2}T^2)^3(1+q^2T+q^4T^2)^3(1+q^2T)^{2}(1-q^2T)^{13}(1-q^3T)}}$} 
				\end{tabular}\\
					
			\hhline{=::=:}

	\begin{tabular}{c@{}c@{}} \setlength{\fboxrule}{2pt}\fcolorbox{white}{white}{$(1,1,3)$}\\
					    \setlength{\fboxrule}{2pt}\fcolorbox{white}{white}{$(1,1,1,1)$}
					\end{tabular} &  
					\begin{tabular}{@{}l@{}}
						\setlength{\fboxrule}{2pt}\fcolorbox{white}{white}{$\displaystyle{\frac{\left(1-(a_q^{3}-3 a_q q)T+q^3 T^2\right)\left(1-a_q qT +q^3T^{2}\right)^{7}\left(1+a_q qT -(q^3-a_q^2 q^2) T^2+a_q q^4 T^3 +q^6 T^4\right)^{}}{(1-T)(1-qT)^{19}(1+qT+q^{2}T^2)(1+q^2T+q^4T^2)(1-q^2T)^{19}(1-q^3T)}}$} 
					\end{tabular} \\
					\hhline{=::=:}
					\begin{tabular}{c@{}c@{}} \setlength{\fboxrule}{2pt}\fcolorbox{white}{white}{$(1,1,3)$}\\
					    \setlength{\fboxrule}{2pt}\fcolorbox{white}{white}{$(1,1,2)$}
					\end{tabular} &
					\begin{tabular}{@{}l@{}} 
						\setlength{\fboxrule}{2pt}\fcolorbox{white}{white}{$\displaystyle{\frac{\left(1-(a_q^{3}-3 a_q q)T+q^3 T^2\right)\left(1-a_q qT +q^3T^{2}\right)^{7}\left(1+a_q qT -(q^3-a_q^2 q^2) T^2+a_q q^4 T^3 +q^6 T^4\right)^{}}{(1-T)(1-qT)^{16}(1+qT)^{3}(1+qT+q^{2}T^2)(1+q^2T+q^4T^2)(1+q^2T)^{3}(1-q^2T)^{16}(1-q^3T)}}$} 
					\end{tabular}\\
				   \hhline{=::=:}
				\begin{tabular}{c@{}c@{}} \setlength{\fboxrule}{2pt}\fcolorbox{white}{white}{$(1,1,3)$}\\
					    \setlength{\fboxrule}{2pt}\fcolorbox{white}{white}{$(1,3)$}
					\end{tabular} &
				\begin{tabular}{@{}l@{}} 
					\setlength{\fboxrule}{2pt}\fcolorbox{white}{white}{$\displaystyle{\frac{\left(1-(a_q^{3}-3 a_q q)T+q^3 T^2\right)\left(1-a_q qT +q^3T^{2}\right)^{7}\left(1+a_q qT -(q^3-a_q^2 q^2) T^2+a_q q^4 T^3 +q^6 T^4\right)^{}}{(1-T)(1-qT)^{13}(1+qT+q^{2}T^2)^4(1+q^2T+q^4T^2)^4(1-q^2T)^{13}(1-q^3T)}}$} 
				\end{tabular}\\
				\hhline{=::=:}
				\begin{tabular}{c@{}c@{}} \setlength{\fboxrule}{2pt}\fcolorbox{white}{white}{$(1,4)$}\\
					    \setlength{\fboxrule}{2pt}\fcolorbox{white}{white}{$(1,1,1,1)$}
					\end{tabular} &  
					\begin{tabular}{@{}l@{}}
						\setlength{\fboxrule}{2pt}\fcolorbox{white}{white}{$\displaystyle{\frac{\left(1-(a_q^{3}-3 a_q q)T+q^3 T^2\right)\left(1-a_q qT +q^3T^{2}\right)^{6}\left(1+a_q qT +q^3T^{2}\right)^{}\left(1+a_q qT -(q^3-a_q^2 q^2) T^2+a_q q^4 T^3 +q^6 T^4\right)^{}}{(1-T)(1-qT)^{18}(1+qT)(1+q^2T^2)(1+q^4T^2)(1+q^2T)(1-q^2T)^{18}(1-q^3T)}}$} 
					\end{tabular} \\
					\hhline{=::=:}
					\begin{tabular}{c@{}c@{}} \setlength{\fboxrule}{2pt}\fcolorbox{white}{white}{$(1,1,3)$}\\
					    \setlength{\fboxrule}{2pt}\fcolorbox{white}{white}{$(1,4)$}
					\end{tabular} &
					\begin{tabular}{@{}l@{}} 
						\setlength{\fboxrule}{2pt}\fcolorbox{white}{white}{$\displaystyle{\frac{\left(1-(a_q^{3}-3 a_q q)T+q^3 T^2\right)\left(1-a_q qT +q^3T^{2}\right)^{6}\left(1+a_q qT +q^3T^{2}\right)^{}\left(1+a_q qT -(q^3-a_q^2 q^2) T^2+a_q q^4 T^3 +q^6 T^4\right)^{}}{(1-T)(1-qT)^{15}(1+qT)^4(1+q^2T^2)(1+q^4T^2)(1+q^2T)^4(1-q^2T)^{15}(1-q^3T)}}$} 
					\end{tabular}\\
				   \hhline{=::=:}
				\begin{tabular}{c@{}c@{}} \setlength{\fboxrule}{2pt}\fcolorbox{white}{white}{$(1,4)$}\\
					    \setlength{\fboxrule}{2pt}\fcolorbox{white}{white}{$(1,3)$}
					\end{tabular} &
				\begin{tabular}{@{}l@{}} 
					\setlength{\fboxrule}{3pt}\fcolorbox{white}{white}{$\displaystyle{\frac{\left(1-(a_q^{3}-3 a_q q)T+q^3 T^2\right)\left(1-a_q qT +q^3T^{2}\right)^{6}\left(1+a_q qT +q^3T^{2}\right)^{}\left(1+a_q qT -(q^3-a_q^2 q^2) T^2+a_q q^4 T^3 +q^6 T^4\right)^{}}{(1-T)(1-qT)^{12}(1+qT)(1+q^2T^2)(1+qT+q^{2}T^2)^{3}(1+q^2T+q^4T^2)^3(1+q^4T^2)(1+q^2T)(1-q^2T)^{12}(1-q^3T)}}$} 
				\end{tabular}\\			
				
\end{tabular}}	
		\end{varwidth}
		\vspace{2mm}
	\end{center}
\end{table}

\textup{     }
\\
\linia
\vspace{1mm}
\begin{center}$G_{12.2}\colon\;\;\left\langle\left(\begin{array}{rrr} 0 & 1 & 0 \\ 0 & 0 & 1 \\ 1 & 0 & 0 \end{array}\right), \left(\begin{array}{rrr} -1 & 0 & 0 \\ 0 & 1 & 0 \\ 0 & 0 & -1 \end{array}\right) \right\rangle\simeq A_{4}$ \end{center}
\vspace{1mm}
\linia
\begin{itemize}
\item[] \textbf{Hodge numbers: } $$h^{1,1}(\kum_{3}(E, G_{12.2}))=19\quad \textup{and} \quad h^{2,1}(\kum_{3}(E, G_{12.2}))=3$$
\item[] \textbf{Zeta function: } 

\begin{table}[H]
	\begin{center}
		\begin{varwidth}{\textheight}
			\resizebox{0.55\textheight}{!}{
				\begin{tabular}{c||c}
				\textrm{Case:}	 &
					\begin{tabular}{@{}l@{}} \setlength{\fboxrule}{1pt}\fcolorbox{white}{white}{$\displaystyle{Z_{q}\left(\kum_{3}(E, G_{12.2})\right)}$}
					\end{tabular} \\
					\hhline{=::=:}
					$(1,1,1,1)$ &  
					\begin{tabular}{@{}l@{}}
						\setlength{\fboxrule}{1pt}\fcolorbox{white}{white}{$\displaystyle{\frac{\left(1-(a_q^{3}-3 a_q q)T+q^3 T^2\right)\left(1-a_q qT +q^3T^{2}\right)^{3}}{(1-T)(1-qT)^{19}(1-q^2T)^{19}(1-q^3T)}}$} 
					\end{tabular} \\
					\hhline{=::=:}
					$(1,1,2)$ &
					\begin{tabular}{@{}l@{}} 
						\setlength{\fboxrule}{1pt}\fcolorbox{white}{white}{$\displaystyle{\frac{\left(1-(a_q^{3}-3 a_q q)T+q^3 T^2\right)\left(1-a_q qT +q^3T^{2}\right)^{3}}{(1-T)(1-qT)^{15}(1+qT)^{4}(1+q^2T)^{4}(1-q^2T)^{15}(1-q^3T)}}$} 
					\end{tabular}\\
				   \hhline{=::=:}
				$(1,3)$ &
				\begin{tabular}{@{}l@{}} 
					\setlength{\fboxrule}{1pt}\fcolorbox{white}{white}{$\displaystyle{\frac{\left(1-(a_q^{3}-3 a_q q)T+q^3 T^2\right)\left(1-a_q qT +q^3T^{2}\right)^{3}}{(1-T)(1-qT)^{9}(1+qT+q^{2}T^2)^5(1+q^2T+q^4T^2)^5(1-q^2T)^{9}(1-q^3T)}}$} 
				\end{tabular}\\
			\end{tabular}}	
		\end{varwidth}
		\vspace{2mm}
	\end{center}
\end{table}
\end{itemize}

\textup{ }
\\
\linia
\vspace{1mm}
\begin{center}$G_{12.3}\colon\;\; \left\langle\left(\begin{array}{rrr} 0 & 1 & 0 \\ 0 & 0 & 1 \\ 1 & 0 & 0 \end{array}\right), \left(\begin{array}{rrr} 0 & -1 & 1 \\ 0 & -1 & 0 \\ 1 & -1 & 0 \end{array}\right) \right\rangle\simeq A_{4}$\end{center} 

\vspace{2mm}
\begin{center}$G_{12.4}\colon\;\; \left\langle\left(\begin{array}{rrr} 0 & 1 & 0 \\ 0 & 0 & 1 \\ 1 & 0 & 0 \end{array}\right), \left(\begin{array}{rrr} -1 & -1 & -1 \\ 0 & 0 & 1 \\ 0 & 1 & 0 \end{array}\right) \right\rangle\simeq A_{4}$ \end{center}
\vspace{1mm}
\linia
\begin{itemize}
\item[] \textbf{Hodge numbers: } \begin{align*}&h^{1,1}(\kum_{3}(E, G_{12.3}))=h^{1,1}(\kum_{3}(E, G_{12.4}))=7 \\
    &h^{2,1}(\kum_{3}(E, G_{12.3}))=h^{2,1}(\kum_{3}(E, G_{12.4}))=3\end{align*}
\item[] \textbf{Zeta function: } 

\begin{table}[H]
	\begin{center}
		\begin{varwidth}{\textheight}
			\resizebox{0.55\textheight}{!}{
				\begin{tabular}{c||c}
				\textrm{Case:}	 &
					\begin{tabular}{@{}l@{}} \setlength{\fboxrule}{1pt}\fcolorbox{white}{white}{$\displaystyle{Z_{q}\left(\kum_{3}(E, G_{12.3})\right)=Z_{q}\left(\kum_{3}(E, G_{12.4})\right)}$}
					\end{tabular} \\
					\hhline{=::=:}
					$(1,1,1,1)$ &  
					\begin{tabular}{@{}l@{}}
						\setlength{\fboxrule}{1pt}\fcolorbox{white}{white}{$\displaystyle{\frac{\left(1-(a_q^{3}-3 a_q q)T+q^3 T^2\right)\left(1-a_q qT +q^3T^{2}\right)^{3}}{(1-T)(1-qT)^{7}(1-q^2T)^{7}(1-q^3T)}}$} 
					\end{tabular} \\
					\hhline{=::=:}
					$(1,1,2)$ &
					\begin{tabular}{@{}l@{}} 
						\setlength{\fboxrule}{1pt}\fcolorbox{white}{white}{$\displaystyle{\frac{\left(1-(a_q^{3}-3 a_q q)T+q^3 T^2\right)\left(1-a_q qT +q^3T^{2}\right)^{3}}{(1-T)(1-qT)^{6}(1+qT)(1+q^2T)(1-q^2T)^{6}(1-q^3T)}}$} 
					\end{tabular}\\
				   \hhline{=::=:}
				$(1,3)$ &
				\begin{tabular}{@{}l@{}} 
					\setlength{\fboxrule}{1pt}\fcolorbox{white}{white}{$\displaystyle{\frac{\left(1-(a_q^{3}-3 a_q q)T+q^3 T^2\right)\left(1-a_q qT +q^3T^{2}\right)^{3}}{(1-T)(1-qT)^{5}(1+qT+q^{2}T^2)(1+q^2T+q^4T^2)(1-q^2T)^{5}(1-q^3T)}}$} 
				\end{tabular}\\
			\end{tabular}}	
		\end{varwidth}
		\vspace{2mm}
	\end{center}
\end{table}
\end{itemize}
\newpage
\subsection{Order 8}
\textup{ }
\\
\linia
\vspace{1mm}
\begin{center}$G_{8.1}\colon\;\; \left\langle\left(\begin{array}{rrr} 1 & 0 & 0 \\ 0 & 0 & -1 \\ 0 & 1 & 0 \end{array}\right), \left(\begin{array}{rrr} -1 & 0 & 0 \\ 0 & 0 & 1 \\ 0 & 1 & 0 \end{array}\right) \right\rangle\simeq D_{4}$\end{center} 
\vspace{1mm}
\linia
\begin{itemize}
\item[] \textbf{Hodge numbers: } $$h^{1,1}(\kum_{3}(E, G_{8.1}))=36\quad \textup{and} \quad h^{2,1}(\kum_{3}(E, G_{8.1}))=6$$
\item[] \textbf{Zeta function: } 

\begin{table}[H]
	\begin{center}
		\begin{varwidth}{\textheight}
			\resizebox{0.55\textheight}{!}{
				\begin{tabular}{c||c}
				\textrm{Case:}	 &
					\begin{tabular}{@{}l@{}} \setlength{\fboxrule}{1pt}\fcolorbox{white}{white}{$\displaystyle{Z_{q}\left(\kum_{3}(E, G_{8.1})\right)}$}
					\end{tabular} \\
					\hhline{=::=:}
					$(1,1,1,1)$ &  
					\begin{tabular}{@{}l@{}}
						\setlength{\fboxrule}{1pt}\fcolorbox{white}{white}{$\displaystyle{\frac{\left(1-(a_q^{3}-3 a_q q)T+q^3 T^2\right)\left(1-a_q qT +q^3T^{2}\right)^{6}}{(1-T)(1-qT)^{36}(1-q^2T)^{36}(1-q^3T)}}$} 
					\end{tabular} \\
					\hhline{=::=:}
					$(1,1,2)$ &
					\begin{tabular}{@{}l@{}} 
						\setlength{\fboxrule}{1pt}\fcolorbox{white}{white}{$\displaystyle{\frac{\left(1-(a_q^{3}-3 a_q q)T+q^3 T^2\right)\left(1+a_q qT +q^3T^{2}\right)\left(1-a_q qT +q^3T^{2}\right)^{5}}{(1-T)(1-qT)^{27}(1+qT)^9(1+q^2T)^9(1-q^2T)^{27}(1-q^3T)}}$} 
					\end{tabular}\\
				   \hhline{=::=:}
				$(1,3)$ &
				\begin{tabular}{@{}l@{}} 
					\setlength{\fboxrule}{1pt}\fcolorbox{white}{white}{$\displaystyle{\frac{\left(1-(a_q^{3}-3 a_q q)T+q^3 T^2\right)\left(1-a_q qT +q^3T^{2}\right)^{4}\left(1+a_q qT -(q^3-a_q^2 q^2) T^2+a_q q^4 T^3 +q^6 T^4\right)^{}}{(1-T)(1-qT)^{18}(1+qT+q^{2}T^2)^9(1+q^2T+q^4T^2)^9(1-q^2T)^{18}(1-q^3T)}
}$} 
				\end{tabular}\\
			\end{tabular}}	
		\end{varwidth}
		\vspace{2mm}
	\end{center}
\end{table}
\end{itemize}

\textup{     }
\\
\linia
\vspace{1mm}
\begin{center}$G_{8.2}\colon\;\;\left\langle\left(\begin{array}{rrr} 1 & 0 & 1 \\ 0 & 0 & -1 \\ 0 & 1 & 0 \end{array}\right), \left(\begin{array}{rrr} -1 & 0 & 0 \\ 0 & 0 & -1 \\ 0 & -1 & 0 \end{array}\right) \right\rangle\simeq D_{4}$\end{center} 
\vspace{1mm}
\linia
\begin{itemize}
\item[] \textbf{Hodge numbers: } $$h^{1,1}(\kum_{3}(E, G_{8.2}))=15\quad \textup{and} \quad h^{2,1}(\kum_{3}(E, G_{8.2}))=3$$
\item[] \textbf{Zeta function: } 
\begin{table}[H]
	\begin{center}
		\begin{varwidth}{\textheight}
			\resizebox{0.55\textheight}{!}{
				\begin{tabular}{c||c}
				\textrm{Case:}	 &
					\begin{tabular}{@{}l@{}} \setlength{\fboxrule}{1pt}\fcolorbox{white}{white}{$\displaystyle{Z_{q}\left(\kum_{3}(E, G_{8.2})\right)}$}
					\end{tabular} \\
					\hhline{=::=:}
					$(1,1,1,1)$ &  
					\begin{tabular}{@{}l@{}}
						\setlength{\fboxrule}{1pt}\fcolorbox{white}{white}{$\displaystyle{\frac{\left(1-(a_q^{3}-3 a_q q)T+q^3 T^2\right)\left(1-a_q qT +q^3T^{2}\right)^{3}}{(1-T)(1-qT)^{15}(1-q^2T)^{15}(1-q^3T)}}$} 
					\end{tabular} \\
					\hhline{=::=:}
					$(1,1,2)$ &
					\begin{tabular}{@{}l@{}} 
						\setlength{\fboxrule}{1pt}\fcolorbox{white}{white}{$\displaystyle{\frac{\left(1-(a_q^{3}-3 a_q q)T+q^3 T^2\right)\left(1-a_q qT +q^3T^{2}\right)^{3}}{(1-T)(1-qT)^{12}(1+qT)^3(1+q^2T)^3(1-q^2T)^{12}(1-q^3T)}}$} 
					\end{tabular}\\
				   \hhline{=::=:}
				$(1,3)$ &
				\begin{tabular}{@{}l@{}} 
					\setlength{\fboxrule}{1pt}\fcolorbox{white}{white}{$\displaystyle{\frac{\left(1-(a_q^{3}-3 a_q q)T+q^3 T^2\right)\left(1-a_q qT +q^3T^{2}\right)^{3}}{(1-T)(1-qT)^{9}(1+qT+q^{2}T^2)^3(1+q^2T+q^4T^2)^3(1-q^2T)^{9}(1-q^3T)}
}$} 
				\end{tabular}\\
			\end{tabular}}	
		\end{varwidth}
		\vspace{2mm}
	\end{center}
\end{table}

\end{itemize}

\newpage
\subsection{Order 6}

\textup{     }
\\
\linia
\vspace{1mm}
\begin{center}$G_{6.1}\colon\;\; \left\langle\left(\begin{array}{rrr} 1 & 0 & 0 \\ 0 & 0 & -1 \\ 0 & 1 & -1 \end{array}\right), \left(\begin{array}{rrr} -1 & 0 & 0 \\ 0 & 0 & -1 \\ 0 & -1 & 0 \end{array}\right) \right\rangle\simeq S_{3}$ \end{center} 
\vspace{2mm}
\begin{center}$G_{6.2}\colon\;\;\left\langle\left(\begin{array}{rrr} 1 & 0 & 0 \\ 0 & 0 & -1 \\ 0 & 1 & -1 \end{array}\right), \left(\begin{array}{rrr} -1 & 0 & 0 \\ 0 & 0 & 1 \\ 0 & 1 & 0 \end{array}\right) \right\rangle\simeq S_{3}$ \end{center}
\vspace{1mm}
\linia
\begin{itemize}

\item[] \textbf{Hodge numbers: } \begin{align*}&h^{1,1}(\kum_{3}(E, G_{6.1}))=h^{1,1}(\kum_{3}(E, G_{6.2}))=15,\;\;\;h^{2,1}(\kum_{3}(E, G_{6.1}))=h^{2,1}(\kum_{3}(E, G_{6.2}))=15\end{align*}

\item[] \textbf{Zeta function: } We get 21 possibilities for the zeta function of $\kum_{3}(E, G_{6.1})$ and $\kum_{3}(E, G_{6.2})$:

$$\frac{\left(1-(a_q^{3}-3 a_q q)T+q^3 T^2\right)\left(1-a_q qT +q^3T^{2}\right)^{2}}{(1-T)(1-qT)^{2}(1-q^2T)^{2}(1-q^3T)}\cdot \left(\left(\frac{1-a_q qT +q^3T^{2}}{(1-qT)(1-q^2T)}\right)\otimes P_{(\eta)}\right)\cdot \left(\left(\frac{1-a_q qT +q^3T^{2}}{(1-qT)(1-q^2T)}\right)\otimes P_{(\lambda)}\right),$$ where
\begin{align*}&(\lambda)\in \left\{(1, 1, 1, 1, 1, 1, 1, 1, 1), (1, 1, 1, 2, 2, 2), (1, 2, 2, 2, 2), (1, 1, 1, 3, 3), (1,4,4), (1,2,6), (1,8)\right\}\\&
(\eta)\in \{(1,1,1,1), (1,1,2), (1,3)\}\end{align*}
\end{itemize}
\vspace{-2mm}
\begin{table}[H]
	\begin{center}
		\begin{varwidth}{\textheight}
			\resizebox{0.65\textheight}{!}{
				\begin{tabular}{c||c}
				\textrm{Case:}	 &
					\begin{tabular}{@{}l@{}} \setlength{\fboxrule}{2pt}\fcolorbox{white}{white}{$\displaystyle{Z_{q}\left(\kum_{3}(E, G_{6.1})\right)=Z_{q}\left(\kum_{3}(E, G_{6.2})\right)}$}
					\end{tabular} \\
					\hhline{=::=:}
					\begin{tabular}{c@{}c@{}} \setlength{\fboxrule}{2pt}\fcolorbox{white}{white}{$(1)^9$}\\
					\setlength{\fboxrule}{2pt}\fcolorbox{white}{white}{$(1,1,1,1)$}
					\end{tabular}
					 &  
					\begin{tabular}{@{}l@{}}
						\setlength{\fboxrule}{2pt}\fcolorbox{white}{white}{$\displaystyle{\frac{\left(1-(a_q^{3}-3 a_q q)T+q^3 T^2\right)\left(1-a_q qT +q^3T^{2}\right)^{15}}{(1-T)(1-qT)^{15}(1-q^2T)^{15}(1-q^3T)}}$} 
					\end{tabular} \\
					\hhline{=::=:}
					\begin{tabular}{c@{}c@{}} \setlength{\fboxrule}{2pt}\fcolorbox{white}{white}{$(1)^9$}\\
					\setlength{\fboxrule}{2pt}\fcolorbox{white}{white}{$(1,1,2)$}
					\end{tabular}
					 &
					\begin{tabular}{@{}l@{}} 
						\setlength{\fboxrule}{2pt}\fcolorbox{white}{white}{$\displaystyle{\frac{\left(1-(a_q^{3}-3 a_q q)T+q^3 T^2\right)\left(1+a_q qT +q^3T^{2}\right)^{}\left(1-a_q qT +q^3T^{2}\right)^{14}}{(1-T)(1-qT)^{14}(1+qT)(1+q^2T)(1-q^2T)^{14}(1-q^3T)}}$} 
					\end{tabular}\\
				   \hhline{=::=:}
				\begin{tabular}{c@{}c@{}} \setlength{\fboxrule}{2pt}\fcolorbox{white}{white}{$(1)^9$}\\
					\setlength{\fboxrule}{2pt}\fcolorbox{white}{white}{$(1,3)$}
					\end{tabular} &
				\begin{tabular}{@{}l@{}} 
					\setlength{\fboxrule}{2pt}\fcolorbox{white}{white}{$\displaystyle{\frac{\left(1-(a_q^{3}-3 a_q q)T+q^3 T^2\right)\left(1-a_q qT +q^3T^{2}\right)^{13}\left(1+a_q qT -(q^3-a_q^2 q^2) T^2+a_q q^4 T^3 +q^6 T^4\right)^{}}{(1-T)(1-qT)^{13}(1+qT+q^{2}T^2)(1+q^2T+q^4T^2)(1-q^2T)^{13}(1-q^3T)}}$} 
				\end{tabular}\\
					\hhline{=::=:}
					\begin{tabular}{c@{}c@{}} \setlength{\fboxrule}{2pt}\fcolorbox{white}{white}{$(1,1,1,2,2,2)$}\\
					\setlength{\fboxrule}{2pt}\fcolorbox{white}{white}{$(1,1,1,1)$}
					\end{tabular}
					 &  
					\begin{tabular}{@{}l@{}}
						\setlength{\fboxrule}{2pt}\fcolorbox{white}{white}{$\displaystyle{\frac{\left(1-(a_q^{3}-3 a_q q)T+q^3 T^2\right)\left(1+a_q qT +q^3T^{2}\right)^{3}\left(1-a_q qT +q^3T^{2}\right)^{12}}{(1-T)(1-qT)^{12}(1+qT)^3(1+q^2T)^3(1-q^2T)^{12}(1-q^3T)}}$} 
					\end{tabular} \\
					\hhline{=::=:}
					\begin{tabular}{c@{}c@{}} \setlength{\fboxrule}{2pt}\fcolorbox{white}{white}{$(1,1,1,2,2,2)$}\\
					\setlength{\fboxrule}{2pt}\fcolorbox{white}{white}{$(1,1,2)$}
					\end{tabular}
					 &
					\begin{tabular}{@{}l@{}} 
						\setlength{\fboxrule}{2pt}\fcolorbox{white}{white}{$\displaystyle{\frac{\left(1-(a_q^{3}-3 a_q q)T+q^3 T^2\right)\left(1+a_q qT +q^3T^{2}\right)^{4}\left(1-a_q qT +q^3T^{2}\right)^{11}}{(1-T)(1-qT)^{11}(1+qT)^4(1+q^2T)^4(1-q^2T)^{11}(1-q^3T)}}$} 
					\end{tabular}\\
				   \hhline{=::=:}
				\begin{tabular}{c@{}c@{}} \setlength{\fboxrule}{2pt}\fcolorbox{white}{white}{$(1,1,1,2,2,2)$}\\
					\setlength{\fboxrule}{2pt}\fcolorbox{white}{white}{$(1,3)$}
					\end{tabular} &
				\begin{tabular}{@{}l@{}} 
					\setlength{\fboxrule}{2pt}\fcolorbox{white}{white}{$\displaystyle{\frac{\left(1-(a_q^{3}-3 a_q q)T+q^3 T^2\right)\left(1-a_q qT +q^3T^{2}\right)^{10}\left(1+a_q qT +q^3T^{2}\right)^{3}\left(1+a_q qT -(q^3-a_q^2 q^2) T^2+a_q q^4 T^3 +q^6 T^4\right)^{}}{(1-T)(1-qT)^{10}(1+qT)^3(1+qT+q^{2}T^2)(1+q^2T+q^4T^2)(1+q^2T)^3(1-q^2T)^{10}(1-q^3T)}}$} 
				\end{tabular}\\
				\hhline{=::=:}
				\begin{tabular}{c@{}c@{}} \setlength{\fboxrule}{2pt}\fcolorbox{white}{white}{$(1,2,2,2,2)$}\\
					\setlength{\fboxrule}{2pt}\fcolorbox{white}{white}{$(1,1,1,1)$}
					\end{tabular}
					 &  
					\begin{tabular}{@{}l@{}}
						\setlength{\fboxrule}{2pt}\fcolorbox{white}{white}{$\displaystyle{\frac{\left(1-(a_q^{3}-3 a_q q)T+q^3 T^2\right)\left(1+a_q qT +q^3T^{2}\right)^{4}\left(1-a_q qT +q^3T^{2}\right)^{11}}{(1-T)(1-qT)^{11}(1+qT)^4(1+q^2T)^4(1-q^2T)^{11}(1-q^3T)}}$} 
					\end{tabular} \\
				\hhline{=::=:}
				\begin{tabular}{c@{}c@{}} \setlength{\fboxrule}{2pt}\fcolorbox{white}{white}{$(1,2,2,2,2)$}\\
					\setlength{\fboxrule}{2pt}\fcolorbox{white}{white}{$(1,1,2)$}
					\end{tabular}
					 &
					\begin{tabular}{@{}l@{}} 
						\setlength{\fboxrule}{2pt}\fcolorbox{white}{white}{$\displaystyle{\frac{\left(1-(a_q^{3}-3 a_q q)T+q^3 T^2\right)\left(1+a_q qT +q^3T^{2}\right)^{5}\left(1-a_q qT +q^3T^{2}\right)^{10}}{(1-T)(1-qT)^{10}(1+qT)^5(1+q^2T)^5(1-q^2T)^{10}(1-q^3T)}}$} 
					\end{tabular}\\
				   \hhline{=::=:}
				\begin{tabular}{c@{}c@{}} \setlength{\fboxrule}{2pt}\fcolorbox{white}{white}{$(1,2,2,2,2)$}\\
					\setlength{\fboxrule}{2pt}\fcolorbox{white}{white}{$(1,3)$}
					\end{tabular} &
				\begin{tabular}{@{}l@{}} 
					\setlength{\fboxrule}{2pt}\fcolorbox{white}{white}{$\displaystyle{\frac{\left(1-(a_q^{3}-3 a_q q)T+q^3 T^2\right)\left(1-a_q qT +q^3T^{2}\right)^{9}\left(1+a_q qT +q^3T^{2}\right)^{4}\left(1+a_q qT -(q^3-a_q^2 q^2) T^2+a_q q^4 T^3 +q^6 T^4\right)^{}}{(1-T)(1-qT)^{9}(1+qT)^4(1+qT+q^{2}T^2)(1+q^2T+q^4T^2)(1+q^2T)^4(1-q^2T)^{9}(1-q^3T)}}$} 
				\end{tabular}\\
				\hhline{=::=:}
				\begin{tabular}{c@{}c@{}} \setlength{\fboxrule}{2pt}\fcolorbox{white}{white}{$(1,1,1,3,3)$}\\
					\setlength{\fboxrule}{2pt}\fcolorbox{white}{white}{$(1,1,1,1)$}
					\end{tabular}
					 &  
					\begin{tabular}{@{}l@{}}
						\setlength{\fboxrule}{2pt}\fcolorbox{white}{white}{$\displaystyle{\frac{\left(1-(a_q^{3}-3 a_q q)T+q^3 T^2\right)\left(1-a_q qT +q^3T^{2}\right)^{11}\left(1+a_q qT -(q^3-a_q^2 q^2) T^2+a_q q^4 T^3 +q^6 T^4\right)^{2}}{(1-T)(1-qT)^{11}(1+qT+q^{2}T^2)^2(1+q^2T+q^4T^2)^2(1-q^2T)^{11}(1-q^3T)}}$} 
					\end{tabular} \\
				\hhline{=::=:}	
				\begin{tabular}{c@{}c@{}} \setlength{\fboxrule}{2pt}\fcolorbox{white}{white}{$(1,1,1,3,3)$}\\
					\setlength{\fboxrule}{2pt}\fcolorbox{white}{white}{$(1,1,2)$}
					\end{tabular}
					 &
					\begin{tabular}{@{}l@{}} 
						\setlength{\fboxrule}{2pt}\fcolorbox{white}{white}{$\displaystyle{\frac{\left(1-(a_q^{3}-3 a_q q)T+q^3 T^2\right)\left(1-a_q qT +q^3T^{2}\right)^{10}\left(1+a_q qT +q^3T^{2}\right)^{}\left(1+a_q qT -(q^3-a_q^2 q^2) T^2+a_q q^4 T^3 +q^6 T^4\right)^{2}}{(1-T)(1-qT)^{10}(1+qT)(1+qT+q^{2}T^2)^2(1+q^2T+q^4T^2)^2(1+q^2T)(1-q^2T)^{10}(1-q^3T)}}$} 
					\end{tabular}\\
				\hhline{=::=:}
					\begin{tabular}{c@{}c@{}} \setlength{\fboxrule}{2pt}\fcolorbox{white}{white}{$(1,1,1,3,3)$}\\
					\setlength{\fboxrule}{2pt}\fcolorbox{white}{white}{$(1,3)$}
					\end{tabular} &
				\begin{tabular}{@{}l@{}} 
					\setlength{\fboxrule}{2pt}\fcolorbox{white}{white}{$\displaystyle{\frac{\left(1-(a_q^{3}-3 a_q q)T+q^3 T^2\right)\left(1-a_q qT +q^3T^{2}\right)^{9}\left(1+a_q qT +q^3T^{2}\right)^{4}\left(1+a_q qT -(q^3-a_q^2 q^2) T^2+a_q q^4 T^3 +q^6 T^4\right)^{}}{(1-T)(1-qT)^{9}(1+qT+q^{2}T^2)^3(1+q^2T+q^4T^2)^3(1-q^2T)^{9}(1-q^3T)}}$} 
				\end{tabular}\\

								\end{tabular}}	
		\end{varwidth}
		\vspace{2mm}
	\end{center}
\end{table}

\begin{table}[H]
	\begin{center}
		\begin{varwidth}{\textheight}
			\resizebox{0.65\textheight}{!}{
				\begin{tabular}{c||c}

					\begin{tabular}{c@{}c@{}} \setlength{\fboxrule}{2pt}\fcolorbox{white}{white}{$(1,4,4)$}\\
					\setlength{\fboxrule}{2pt}\fcolorbox{white}{white}{$(1,1,1,1)$}
					\end{tabular}
					 &  
					\begin{tabular}{@{}l@{}}
						\setlength{\fboxrule}{2pt}\fcolorbox{white}{white}{$\displaystyle{\frac{\left(1-(a_q^{3}-3 a_q q)T+q^3 T^2\right)\left(1-a_q qT +q^3T^{2}\right)^{9}\left(1+a_q qT +q^3T^{2}\right)^{2}\left(1+a_q qT -(q^3-a_q^2 q^2) T^2+a_q q^4 T^3 +q^6 T^4\right)^{2}}{(1-T)(1-qT)^{9}(1+qT)^2(1+q^2T^2)^2(1+q^4T^2)^2(1+q^2T)^2(1-q^2T)^{9}(1-q^3T)}}$} 
					\end{tabular} \\
					\hhline{=::=:}
					\begin{tabular}{c@{}c@{}} \setlength{\fboxrule}{2pt}\fcolorbox{white}{white}{$(1,4,4)$}\\
					\setlength{\fboxrule}{2pt}\fcolorbox{white}{white}{$(1,1,2)$}
					\end{tabular}
					 &
					\begin{tabular}{@{}l@{}} 
						\setlength{\fboxrule}{2pt}\fcolorbox{white}{white}{$\displaystyle{\frac{\left(1-(a_q^{3}-3 a_q q)T+q^3 T^2\right)\left(1-a_q qT +q^3T^{2}\right)^{8}\left(1+a_q qT +q^3T^{2}\right)^{3}\left(1+a_q qT -(q^3-a_q^2 q^2) T^2+a_q q^4 T^3 +q^6 T^4\right)^{2}}{(1-T)(1-qT)^{8}(1+qT)^3(1+q^2T^2)^2(1+q^4T^2)^2(1+q^2T)^3(1-q^2T)^{8}(1-q^3T)}}$} 
					\end{tabular}\\
				   \hhline{=::=:}
				\begin{tabular}{c@{}c@{}} \setlength{\fboxrule}{2pt}\fcolorbox{white}{white}{$(1,4,4)$}\\
					\setlength{\fboxrule}{2pt}\fcolorbox{white}{white}{$(1,3)$}
					\end{tabular} &
				\begin{tabular}{@{}l@{}} 
					\setlength{\fboxrule}{2pt}\fcolorbox{white}{white}{$\displaystyle{\frac{\left(1-(a_q^{3}-3 a_q q)T+q^3 T^2\right)\left(1-a_q qT +q^3T^{2}\right)^{7}\left(1+a_q qT +q^3T^{2}\right)^{2}\left(1+a_q qT -(q^3-a_q^2 q^2) T^2+a_q q^4 T^3 +q^6 T^4\right)^{3}}{(1-T)(1-qT)^{7}(1+qT)^2(1+q^2T^2)^2(1+qT+q^{2}T^2)(1+q^2T+q^4T^2)(1+q^4T^2)^2(1+q^2T)^2(1-q^2T)^{7}(1-q^3T)}}$} 
				\end{tabular}\\
				\hhline{=::=:}
					\begin{tabular}{c@{}c@{}} \setlength{\fboxrule}{2pt}\fcolorbox{white}{white}{$(1,2,6)$}\\
					\setlength{\fboxrule}{2pt}\fcolorbox{white}{white}{$(1,1,1,1)$}
					\end{tabular}
					 &  
					\begin{tabular}{@{}l@{}}
						\setlength{\fboxrule}{2pt}\fcolorbox{white}{white}{$
						\begin{aligned}&\displaystyle{\frac{\left(1-(a_q^{3}-3 a_q q)T+q^3 T^2\right)\left(1-a_q qT +q^3T^{2}\right)^{9}\left(1+a_q qT -(q^3-a_q^2 q^2) T^2+a_q q^4 T^3 +q^6 T^4\right)^{}}{(1-T)(1-qT)^{9}(1+qT)^{2}(1+qT+q^{2}T^2)(1-qT+q^{2}T^2)(1-q^2T+q^4T^2)}}\times \\ &\times \frac{\left(1+a_q qT +q^3T^{2}\right)^{2}\left(1-a_q qT -(q^3-a_q^2 q^2) T^2-a_q q^4 T^3 +q^6 T^4\right)^{}}{(1+q^2T+q^4T^2)(1+q^2T)^{2}(1-q^2T)^{9}(1-q^3T)}\end{aligned}$} 
					\end{tabular} \\
	\hhline{=::=:}

				\begin{tabular}{c@{}c@{}} \setlength{\fboxrule}{2pt}\fcolorbox{white}{white}{$(1,2,6)$}\\
					\setlength{\fboxrule}{2pt}\fcolorbox{white}{white}{$(1,1,2)$}
					\end{tabular}
					 &
					\begin{tabular}{@{}l@{}} 
						\setlength{\fboxrule}{2pt}\fcolorbox{white}{white}{$\begin{aligned}&\displaystyle{\frac{\left(1-(a_q^{3}-3 a_q q)T+q^3 T^2\right)\left(1-a_q qT +q^3T^{2}\right)^{8}\left(1+a_q qT -(q^3-a_q^2 q^2) T^2+a_q q^4 T^3 +q^6 T^4\right)^{}}{(1-T)(1-qT)^{8}(1+qT)^{3}(1+qT+q^{2}T^2)(1-qT+q^{2}T^2)(1-q^2T+q^4T^2)}}\times \\ &\times \frac{\left(1+a_q qT +q^3T^{2}\right)^{3}\left(1-a_q qT -(q^3-a_q^2 q^2) T^2-a_q q^4 T^3 +q^6 T^4\right)^{}}{(1+q^2T+q^4T^2)(1+q^2T)^{3}(1-q^2T)^{8}(1-q^3T)}\end{aligned}$} 
					\end{tabular}\\	
				\hhline{=::=:}

\begin{tabular}{c@{}c@{}} \setlength{\fboxrule}{2pt}\fcolorbox{white}{white}{$(1,2,6)$}\\
					\setlength{\fboxrule}{2pt}\fcolorbox{white}{white}{$(1,3)$}
					\end{tabular} &	
				\begin{tabular}{@{}l@{}} 
					\setlength{\fboxrule}{2pt}\fcolorbox{white}{white}{$\begin{aligned}&\displaystyle{\frac{\left(1-(a_q^{3}-3 a_q q)T+q^3 T^2\right)\left(1-a_q qT +q^3T^{2}\right)^{7}\left(1+a_q qT -(q^3-a_q^2 q^2) T^2+a_q q^4 T^3 +q^6 T^4\right)^{2}}{(1-T)(1-qT)^{7}(1+qT)^{2}(1+qT+q^{2}T^2)^2(1-qT+q^{2}T^2)(1-q^2T+q^4T^2)}}\times \\ & \times \frac{\left(1+a_q qT +q^3T^{2}\right)^{2}\left(1-a_q qT -(q^3-a_q^2 q^2) T^2-a_q q^4 T^3 +q^6 T^4\right)^{}}{(1+q^2T+q^4T^2)^2(1+q^2T)^{2}(1-q^2T)^{7}(1-q^3T)}\end{aligned}$} 
				\end{tabular}\\
				\hhline{=::=:}
					\begin{tabular}{c@{}c@{}} \setlength{\fboxrule}{2pt}\fcolorbox{white}{white}{$(1,8)$}\\
					\setlength{\fboxrule}{2pt}\fcolorbox{white}{white}{$(1,1,1,1)$}
					\end{tabular}
					 &  
					\begin{tabular}{@{}l@{}}
						\setlength{\fboxrule}{2pt}\fcolorbox{white}{white}{$\begin{aligned}&\displaystyle{\frac{\left(1-(a_q^{3}-3 a_q q)T+q^3 T^2\right)\left(1-a_q qT +q^3T^{2}\right)^{8}\left(1+a_q qT +q^3T^{2}\right)^{}\left(1-2q^3T^2+a_q^2 q^2 T^2 +q^6 T^4\right)^{}}{(1-T)(1-qT)^{8}(1+qT)^{}(1+q^2T^{2})^{}(1+q^4T^2)(1+q^4T^4)}}\times \\ &\times \frac{\left(1+2q^6 T^4 -4a_q^2 q^5 T^4+a_q^4 q^4 T^4 +q^{12}T^8\right)^{}}{(1+q^8T^4)(1+q^4T^2)^{}(1+q^2T)^{}(1-q^2T)^{8}(1-q^3T)}\end{aligned}$} 
					\end{tabular} \\
					\hhline{=::=:}
					\begin{tabular}{c@{}c@{}} \setlength{\fboxrule}{2pt}\fcolorbox{white}{white}{$(1,8)$}\\
					\setlength{\fboxrule}{2pt}\fcolorbox{white}{white}{$(1,1,2)$}
					\end{tabular}
					 &
					\begin{tabular}{@{}l@{}} 
						\setlength{\fboxrule}{2pt}\fcolorbox{white}{white}{$\begin{aligned}&\displaystyle{\frac{\left(1-(a_q^{3}-3 a_q q)T+q^3 T^2\right)\left(1-a_q qT +q^3T^{2}\right)^{7}\left(1+a_q qT +q^3T^{2}\right)^{2}\left(1-2q^3T^2+a_q^2 q^2 T^2 +q^6 T^4\right)^{}}{(1-T)(1-qT)^{7}(1+qT)^{2}(1+q^2T^{2})^{}(1+q^4T^2)(1+q^4T^4)}}\times \\ &\times \frac{\left(1+2q^6 T^4 -4a_q^2 q^5 T^4+a_q^4 q^4 T^4 +q^{12}T^8\right)^{}}{(1+q^8T^4)(1+q^4T^2)^{}(1+q^2T)^{2}(1-q^2T)^{7}(1-q^3T)}\end{aligned}$} 
					\end{tabular}\\
				   \hhline{=::=:}
				\begin{tabular}{c@{}c@{}} \setlength{\fboxrule}{2pt}\fcolorbox{white}{white}{$(1,8)$}\\
					\setlength{\fboxrule}{2pt}\fcolorbox{white}{white}{$(1,3)$}
					\end{tabular} &
				\begin{tabular}{@{}l@{}} 
					\setlength{\fboxrule}{2pt}\fcolorbox{white}{white}{$\begin{aligned}&\displaystyle{\frac{\left(1-(a_q^{3}-3 a_q q)T+q^3 T^2\right)\left(1-a_q qT +q^3T^{2}\right)^{6}\left(1+a_q qT +q^3T^{2}\right)^{}\left(1-2q^3T^2+a_q^2 q^2 T^2 +q^6 T^4\right)^{}}{(1-T)(1-qT)^{6}(1+qT)^{}(1+q^2T^{2})^{}(1+q^4T^2)(1+qT+q^{2}T^2)(1+q^4T^4)}}\times\\&\times\frac{\left(1+a_q qT -(q^3-a_q^2 q^2) T^2+a_q q^4 T^3 +q^6 T^4\right)^{}\left(1+2q^6 T^4 -4a_q^2 q^5 T^4+a_q^4 q^4 T^4 +q^{12}T^8\right)^{}}{(1+q^8T^4)(1+q^2T+q^4T^2)(1+q^4T^2)^{}(1+q^2T)^{}(1-q^2T)^{6}(1-q^3T)}\end{aligned}$} 
				\end{tabular}\\
			\end{tabular}}	
		\end{varwidth}
		\vspace{2mm}
	\end{center}
\end{table}

\textup{     }
\\
\linia
\vspace{1mm}
\begin{center}$G_{6.3}\colon\;\; \left\langle\left(\begin{array}{rrr} 0 & 1 & 0 \\ 0 & 0 & 1 \\ 1 & 0 & 0 \end{array}\right), \left(\begin{array}{rrr} 0 & 0 & -1 \\ 0 & -1 & 0 \\ -1 & 0 & 0 \end{array}\right) \right\rangle\simeq S_{3}$ \end{center} 
\vspace{1mm}
\linia
\begin{itemize}
\item[] \textbf{Hodge numbers: } $$h^{1,1}(\kum_{3}(E, G_{6.3}))=7\quad \textup{and} \quad h^{2,1}(\kum_{3}(E, G_{6.3}))=7$$

\item[] \textbf{Zeta function: } 

\begin{table}[H]
	\begin{center}
		\begin{varwidth}{\textheight}
			\resizebox{0.55\textheight}{!}{
				\begin{tabular}{c||c}
				\textrm{Case:}	 &
					\begin{tabular}{@{}l@{}} \setlength{\fboxrule}{1pt}\fcolorbox{white}{white}{$\displaystyle{Z_{q}\left(\kum_{3}(E, G_{6.3})\right)}$}
					\end{tabular} \\
					\hhline{=::=:}
					$(1,1,1,1)$ &  
					\begin{tabular}{@{}l@{}}
						\setlength{\fboxrule}{1pt}\fcolorbox{white}{white}{$\displaystyle{\frac{\left(1-(a_q^{3}-3 a_q q)T+q^3 T^2\right)\left(1-a_q qT +q^3T^{2}\right)^{7}}{(1-T)(1-qT)^{7}(1-q^2T)^{7}(1-q^3T)}}$} 
					\end{tabular} \\
					\hhline{=::=:}
					$(1,1,2)$ &
					\begin{tabular}{@{}l@{}} 
						\setlength{\fboxrule}{1pt}\fcolorbox{white}{white}{$\displaystyle{\frac{\left(1-(a_q^{3}-3 a_q q)T+q^3 T^2\right)\left(1-a_q qT +q^3T^{2}\right)^{3}}{(1-T)(1-qT)^{12}(1+qT)^3(1+q^2T)^3(1-q^2T)^{12}(1-q^3T)}}$} 
					\end{tabular}\\
				   \hhline{=::=:}
				$(1,3)$ &
				\begin{tabular}{@{}l@{}} 
					\setlength{\fboxrule}{1pt}\fcolorbox{white}{white}{$\displaystyle{\frac{\left(1-(a_q^{3}-3 a_q q)T+q^3 T^2\right)\left(1-a_q qT +q^3T^{2}\right)^{5}\left(1+a_q qT -(q^3-a_q^2 q^2) T^2+a_q q^4 T^3 +q^6 T^4\right)^{}}{(1-T)(1-qT)^{5}(1+qT+q^{2}T^2)(1+q^2T+q^4T^2)(1-q^2T)^{5}(1-q^3T)}
}$} 
				\end{tabular}\\
			\end{tabular}}	
		\end{varwidth}
		\vspace{2mm}
	\end{center}
\end{table}
\end{itemize}

\subsection{Order 4}

\textup{     }
\\
\linia
\vspace{1mm}
\begin{center}$G_{4.1}\colon\;\; \left\langle\left(\begin{array}{rrr} 1 & 0 & 0 \\ 0 & -1 & 0 \\ 0 & 0 & -1 \end{array}\right), \left(\begin{array}{rrr} -1 & 0 & 0 \\ 0 & -1 & 0 \\ 0 & 0 & 1 \end{array}\right) \right\rangle\simeq \ZZ_2\oplus \ZZ_2$\end{center} 
\vspace{1mm}
\linia
\vspace{5mm}
\begin{itemize}
\item[] \textbf{Hodge numbers: } $$h^{1,1}\left(\kum_{3}(E, G_{4.1})\right)=51\quad \textup{and} \quad h^{2,1}(\kum_{3}(E, G_{4.1}))=3$$
\item[] \textbf{Zeta function: }

\begin{table}[H]
	\begin{center}
		\begin{varwidth}{\textheight}
			\resizebox{0.55\textheight}{!}{
				\begin{tabular}{c||c}
				\textrm{Case:}	 &
					\begin{tabular}{@{}l@{}} \setlength{\fboxrule}{1pt}\fcolorbox{white}{white}{$\displaystyle{Z_{q}\left(\kum_{3}(E, G_{4.1})\right)}$}
					\end{tabular} \\
					\hhline{=::=:}
					$(1,1,1,1)$ &  
					\begin{tabular}{@{}l@{}}
						\setlength{\fboxrule}{1pt}\fcolorbox{white}{white}{$\displaystyle{\frac{\left(1-(a_q^{3}-3 a_q q)T+q^3 T^2\right)\left(1-a_q qT +q^3T^{2}\right)^{3}}{(1-T)(1-qT)^{51}(1-q^2T)^{51}(1-q^3T)}}$} 
					\end{tabular} \\
					\hhline{=::=:}
					$(1,1,2)$ &
					\begin{tabular}{@{}l@{}} 
						\setlength{\fboxrule}{1pt}\fcolorbox{white}{white}{$\displaystyle{\frac{\left(1-(a_q^{3}-3 a_q q)T+q^3 T^2\right)\left(1-a_q qT +q^3T^{2}\right)^{3}}{(1-T)(1-qT)^{39}(1+qT)^{12}(1+q^2T)^{12}(1-q^2T)^{39}(1-q^3T)}}$} 
					\end{tabular}\\
				   \hhline{=::=:}
				$(1,3)$ &
				\begin{tabular}{@{}l@{}} 
					\setlength{\fboxrule}{1pt}\fcolorbox{white}{white}{$\displaystyle{\frac{\left(1-(a_q^{3}-3 a_q q)T+q^3 T^2\right)\left(1-a_q qT +q^3T^{2}\right)^{3}}{(1-T)(1-qT)^{21}(1+qT+q^{2}T^2)^{15}(1+q^2T+q^4T^2)^{15}(1-q^2T)^{21}(1-q^3T)}
}$} 
				\end{tabular}\\
			\end{tabular}}	
		\end{varwidth}
		\vspace{2mm}
	\end{center}
\end{table}
\end{itemize}

\textup{     }
\\
\linia
\vspace{1mm}
\begin{center}$G_{4.2}\colon\;\; \left\langle\left(\begin{array}{rrr} 1 & 0 & 0 \\ 0 & -1 & 0 \\ 0 & 0 & -1 \end{array}\right), \left(\begin{array}{rrr} -1 & 0 & 0 \\ 0 & 0 & -1 \\ 0 & -1 & 0 \end{array}\right) \right\rangle\simeq \ZZ_2\oplus \ZZ_2$ \end{center}
\vspace{1mm}

\linia
\begin{itemize}
\item[] \textbf{Hodge numbers: } $$h^{1,1}(\kum_{3}(E, G_{4.2}))=21\quad \textup{and} \quad h^{2,1}(\kum_{3}(E, G_{4.2}))=9$$
\item[] \textbf{Zeta function: } 

\begin{table}[H]
	\begin{center}
		\begin{varwidth}{\textheight}
			\resizebox{0.55\textheight}{!}{
				\begin{tabular}{c||c}
				\textrm{Case:}	 &
					\begin{tabular}{@{}l@{}} \setlength{\fboxrule}{1pt}\fcolorbox{white}{white}{$\displaystyle{Z_{q}\left(\kum_{3}(E, G_{4,2})\right)}$}
					\end{tabular} \\
					\hhline{=::=:}
					$(1,1,1,1)$ &  
					\begin{tabular}{@{}l@{}}
						\setlength{\fboxrule}{1pt}\fcolorbox{white}{white}{$\displaystyle{\frac{\left(1-(a_q^{3}-3 a_q q)T+q^3 T^2\right)\left(1-a_q qT +q^3T^{2}\right)^{9}}{(1-T)(1-qT)^{21}(1-q^2T)^{21}(1-q^3T)}}$} 
					\end{tabular} \\
					\hhline{=::=:}
					$(1,1,2)$ &
					\begin{tabular}{@{}l@{}} 
						\setlength{\fboxrule}{1pt}\fcolorbox{white}{white}{$\displaystyle{\frac{\left(1-(a_q^{3}-3 a_q q)T+q^3 T^2\right)\left(1-a_q qT +q^3T^{2}\right)^{9}}{(1-T)(1-qT)^{18}(1+qT)^{3}(1+q^2T)^{3}(1-q^2T)^{18}(1-q^3T)}}$} 
					\end{tabular}\\
				   \hhline{=::=:}
				$(1,3)$ &
				\begin{tabular}{@{}l@{}} 
					\setlength{\fboxrule}{1pt}\fcolorbox{white}{white}{$\displaystyle{\frac{\left(1-(a_q^{3}-3 a_q q)T+q^3 T^2\right)\left(1-a_q qT +q^3T^{2}\right)^{9}}{(1-T)(1-qT)^{15}(1+qT+q^{2}T^2)^{3}(1+q^2T+q^4T^2)^{3}(1-q^2T)^{15}(1-q^3T)}}$}
				\end{tabular}\\
			\end{tabular}}	
		\end{varwidth}
		\vspace{2mm}
	\end{center}
\end{table}
\end{itemize}

\newpage
\textup{     }
\\
\linia
\vspace{1mm}
\begin{center}$G_{4.3}\colon\;\;\left\langle\left(\begin{array}{rrr} -1 & 0 & 0 \\ 0 & 0 & 1 \\ 0 & 1 & 0 \end{array}\right), \left(\begin{array}{rrr} -1 & 0 & 0 \\ 1 & 0 & -1 \\ -1 & -1 & 0 \end{array}\right) \right\rangle\simeq \ZZ_2\oplus \ZZ_2$\end{center} 
\vspace{2mm}
\begin{center}$G_{4.4}\colon\;\; \left\langle\left(\begin{array}{rrr} -1 & 0 & 0 \\ 0 & 0 & 1 \\ 0 & 1 & 0 \end{array}\right), \left(\begin{array}{rrr} -1 & 1 & -1 \\ 0 & 0 & -1 \\ 0 & -1 & 0 \end{array}\right) \right\rangle\simeq \ZZ_2\oplus \ZZ_2$\end{center} 
\vspace{1mm}
\linia

\begin{itemize}
\item[] \textbf{Hodge numbers: } \begin{align*}&h^{1,1}(\kum_{3}(E, G_{4.3}))=h^{1,1}(\kum_{3}(E, G_{4.4}))=15 \\
    &h^{2,1}(\kum_{3}(E, G_{4.3}))=h^{2,1}(\kum_{3}(E, G_{4.4}))=3\end{align*}
\item[] \textbf{Zeta function: } 

\begin{table}[H]
	\begin{center}
		\begin{varwidth}{\textheight}
			\resizebox{0.55\textheight}{!}{
				\begin{tabular}{c||c}
				\textrm{Case:}	 &
					\begin{tabular}{@{}l@{}} \setlength{\fboxrule}{1pt}\fcolorbox{white}{white}{$\displaystyle{Z_{q}\left(\kum_{3}(E, G_{4,3})\right)=Z_{q}\left(\kum_{3}(E, G_{4,4})\right)}$}
					\end{tabular} \\
					\hhline{=::=:}
					$(1,1,1,1)$ &  
					\begin{tabular}{@{}l@{}}
						\setlength{\fboxrule}{1pt}\fcolorbox{white}{white}{$\displaystyle{\frac{\left(1-(a_q^{3}-3 a_q q)T+q^3 T^2\right)\left(1-a_q qT +q^3T^{2}\right)^{3}}{(1-T)(1-qT)^{15}(1-q^2T)^{15}(1-q^3T)}}$} 
					\end{tabular} \\
					\hhline{=::=:}
					$(1,1,2)$ &
					\begin{tabular}{@{}l@{}} 
						\setlength{\fboxrule}{1pt}\fcolorbox{white}{white}{$\displaystyle{\frac{\left(1-(a_q^{3}-3 a_q q)T+q^3 T^2\right)\left(1-a_q qT +q^3T^{2}\right)^{3}}{(1-T)(1-qT)^{12}(1+qT)^{3}(1+q^2T)^{3}(1-q^2T)^{12}(1-q^3T)}}$} 
					\end{tabular}\\
				   \hhline{=::=:}
				$(1,3)$ &
				\begin{tabular}{@{}l@{}} 
					\setlength{\fboxrule}{1pt}\fcolorbox{white}{white}{$\displaystyle{\frac{\left(1-(a_q^{3}-3 a_q q)T+q^3 T^2\right)\left(1-a_q qT +q^3T^{2}\right)^{3}}{(1-T)(1-qT)^{9}(1+qT+q^{2}T^2)^{3}(1+q^2T+q^4T^2)^{3}(1-q^2T)^{9}(1-q^3T)}}$}
				\end{tabular}\\
			\end{tabular}}	
		\end{varwidth}
		\vspace{2mm}
	\end{center}
\end{table}
\end{itemize}

\renewcommand*{\bibfont}{\fontsize{10}{11}\selectfont}
\printbibliography[heading=bibintoc,title={References}]
\end{document}